\documentclass[11pt,twoside,reqno,centertags]{amsart}
\usepackage{amsmath,amsthm,amsfonts,amssymb}
\advance\textheight by 6truemm
\newtheorem{theorem}{Theorem}[section]
\newtheorem{corollary}{Corollary}[section]
\newtheorem{proposition}{Proposition}[section]

\newtheorem{remark}{Remark}[section]

\usepackage{abraces}

\usepackage{fancybox,color}
\newcommand{\fred}{\color{red}}   
\newcommand{\fblue}{\color{blue}} 

\newcommand{\be}{\begin{equation} \label}
\newcommand{\ee}{\end{equation}}
\newcommand{\R}{{\mathbb R}}

\newcommand{\eps}{\varepsilon}

\numberwithin{equation}{section}

\begin{document}
\title[Oscillatory blow-up]
{Oscillatory blow-up and gradient estimates for semilinear heat equations}
	
\author[Pavol Quittner and Philippe Souplet]{Pavol Quittner$^{(1)}$ and Philippe Souplet$^{(2)}$}

\thanks{$^{(1)}$Department of Applied Mathematics and Statistics, Comenius University
Mlynsk\'a dolina, 84248 Bratislava, Slovakia. Email: quittner@fmph.uniba.sk} 

\thanks{$^{(2)}$Universit\'e Sorbonne Paris Nord, CNRS UMR 7539, LAGA,
93430 Villetaneuse, France. Email: souplet@math.univ-paris13.fr}

\date{}

\begin{abstract}
For reaction-diffusion with blow-up nonlinearities, we consider the question whether
the sup norm of any positive blow-up solution must be eventually monotone nondecreasing in time.
While some sufficient conditions are known, especially for radial solutions,
this natural and basic question for the blow-up theory does not 
seem to have been addressed so far in full generality. 

We construct surprising (nonradial) counter-examples of blow-up solutions with oscillatory $L^\infty$ norm, 
for any Sobolev supercritical power nonlinearity, which show that this property may fail.
In addition, this provides examples of type II blow-up for any supercritical power,
which considerably increases the known range of powers for which type II blow-up may occur.
Moreover, whereas all the type II blow-up rates known so far were at most polynomial,
the blow-up in our counter-examples can be arbitrarily singular.

As a related question, we clarify the gradient estimates obtained and used in previous works.
In particular we show that these estimates hold only at times when the $L^\infty$ norm 
is maximal with respect to the past.


\end{abstract}

\maketitle

\section{Introduction}
Let $\Omega$ be a uniformly smooth domain of $\mathbb{R}^n$ $(n\ge 1)$. 
We consider classical solutions of the semilinear heat equation: 
\be{eqE1} \begin{cases}
 u_t-\Delta u=f(u),&x\in \Omega,\ t>0,\\
 u=0,&x\in \partial\Omega,\  t>0,\\
 u(x,0)=u_0(x),&x \in \Omega,
\end{cases}
\ee
where  the boundary condition is omitted if $\Omega=\R^n$, and
\be{hypf}
 \hbox{$f\in C^1([0,\infty))$ 
  satisfies $f(0)\ge 0$ and $f(s)>0$ for $s>0$,\quad $u_0\in L^\infty(\Omega)$, $u_0\ge 0$.}
\ee 
Throughout this paper we will often use the notation
$$U(t):=\|u(t)\|_\infty,$$ 
where $u(t)=u(\cdot,t)$,
and we will also denote by $p_S$ the Sobolev exponent (i.e., $p_S=(n+2)/(n-2)$ for $n\ge 3$,
$p_S=\infty$ for $n\le 2$).
It is well known 
		that problem \eqref{eqE1} admits a unique, nonnegative maximal $L^\infty$-solution
		(see the beginning of Section~\ref{SecProofs} for precise definition).
Denoting by $T=T(u_0)\in(0,\infty]$ its maximal existence time, we have
	\be{BUalt}
	\hbox{either $T=\infty$, \ or else $\displaystyle\lim_{t\to T} U(t)=\infty$.}
\ee

\section{Monotonicity vs.~oscillation for the sup norm}
\subsection{Main results}
 In this article we are primarily interested in the following general question for
 blow-up solutions of problem \eqref{eqE1}:
\be{questionUmonot}
\hbox{Does $U(t)$ become nondecreasing close to the blow-up time ?}
\ee
and with its connections with other central topics in nonlinear parabolic equations, 
namely type I/II blow-up and gradient estimates.
We note that a positive answer to question \eqref{questionUmonot} would seem reasonable from a heuristic point of view.
Indeed, the occurence of blow-up indicates that, in some sense, the reaction term will eventually 
dominate the dissipative effect of the diffusion near maximum points, 
and solutions of the diffusion-free problem are time increasing owing to $f>0$.
While there are some sufficient conditions (see Remark~\ref{rem1}(i) below), this natural and basic question for the blow-up theory
has not been addressed so far in full generality,
and no counter-example seems to be known.
The first goal of this paper is to provide counter-examples showing that, somehow surprisingly,
the answer to question \eqref{questionUmonot} can be negative.

\begin{theorem} \label{thmNonMon}
Let $\Omega=\R^n$, $n\ge 3$, and $f(u)=u^p$ with $p>p_S$.
There exists $u_0\in L^\infty(\R^n)$, $u_0\ge 0$, such that
the classical solution of \eqref{eqE1}
undergoes nonmonotone $L^\infty$-blow-up at~$T=T(u_0)<\infty$.

More precisely,
there exists an increasing sequence $N_k\nearrow\infty$ 
with the following property: For any increasing sequence $\{M_k\}$ satisfying $M_k>N_k$
there exist $u_0\in L^\infty(\R^n)$, $u_0\ge 0$, and
an increasing sequence $t_k\nearrow T$ such that
\be{nonmonotBU}
\|u(\cdot,t_{2k})\|_\infty<N_k<M_k<\|u(\cdot,t_{2k-1})\|_\infty<M_k+1.
\ee
\end{theorem}

\medskip

Next recall that blow-up for problem \eqref{eqE1} with $f(u)=u^p$ is said to be type I if
\be{typeIbu}
\limsup_{t\to T} \ (T-t)^\beta\|u(\cdot,t)\|_\infty <\infty, \quad \beta:=\textstyle\frac{1}{p-1},
\ee
 and type II otherwise, 
and that blow-up is known to be type I whenever $p<p_S$
(see \cite[Section~23.2]{QS} and the references therein, and \cite{Q21,MTZ}).
It turns out that, for any $p>p_S$, the oscillatory solution in Theorem~\ref{thmNonMon} 
can be chosen in such a way that blow-up is of type II.

\begin{corollary}
\label{corNonMon}
Let $\Omega=\R^n$, $n\ge 3$, $f(u)=u^p$ with $p>p_S$, and let
$$\hbox{$\phi:(0,\infty)\to (0,\infty)$ be an increasing function with $\lim_{s\to\infty} s^{-1}\phi(s)=\infty$.}$$
Then there exist $u_0\in L^\infty(\R^n)$, $u_0\ge 0$ and a sequence $\tau_k\nearrow T=T(u_0)<\infty$, such that
the classical solution of \eqref{eqE1} satisfies 
$\|u(\cdot,\tau_k)\|_\infty\ge \phi\big((T-\tau_k)^{-\beta}\big)$.
In particular $u$ undergoes type II blow-up at $T<\infty$.
\end{corollary}

Corollary~\ref{corNonMon} considerably increases the known range of powers for which type II blow-up may occur.
Moreover, whereas all the rates known so far are at most polynomial,
blow-up in Corollary~\ref{corNonMon} can be arbitrarily singular.
Namely, examples of type II blow-up for positive solutions (either in $\R^n$ or in a bounded domain) 
were so far known only for 
dimensions $n\ge 11$ with $p\ge p_{JL}$,
where 
\be{pJL}
   p_{JL}:=\begin{cases}
         \infty & \hbox{if }n\leq10, \\ 
         1+4\frac{n-4+2\sqrt{n-1}}{(n-2)(n-10)} & \hbox{if }n>10
        \end{cases} 
\ee 
is the Joseph-Lundgren exponent 
(see \cite{HV94a, HV94b, M04, M09, MM11, Co17, Se18, CMR20}), 
for $n\ge 7$ and $p=(n+1)/(n-3)$ (see \cite{DMW21}), or for $n\ge 5$ and $p=3$ (see \cite{DMWZ26}).
As for sign-changing solutions (with $f(u)=|u|^{p-1}u$), some examples are known
for $p=p_S$ with $3\le n\le 6$ (see \cite{schweyer_jfa12, dpmw_ams19, dpmwz_asnap20, dpmwzz20, Har1, Har2}).
We refer to the above works and to \cite{Mat07,M07,M11} for results on the type II blow-up rates, 
and to  \cite[Section 23.2]{QS} for an introduction to the topic of type~I/II blow-up.

As a key difference with the above works, our solutions undergo blow-up at space infinity, rather than at a finite point.
We refer to \cite{GU06,Shi08} for results on other topics related to blow-up at space infinity,
such as sufficient conditions on initial data and spatial bounds on solutions.
 On the other hand, for any $p>p_S$, 
the proof of Theorem~\ref{thmNonMon} also provides solutions with oscillatory blow-up of type I;
see Remark~\ref{remTypeI}.

\begin{remark} \label{remGlobal} \rm
For {\it global} classical solutions of \eqref{eqE1}, 
the lack of eventual (increasing or decreasing) monotonicity of $U(t)$ of certain solutions has been known before.
Namely, for $\Omega=\R^n$ and $f(u)=u^p$ with $p>p_S$ 
(see \cite{Q17} and cf.~also \cite{PY1, PY2, PY3}),
there exists $u_0\ge 0$ bounded, continuous and radially symmetric nonincreasing,
such that the solution of \eqref{eqE1} is global and satisfies
\be{supinf}
\liminf_{t\to\infty}\|u(t)\|_\infty=0 \qquad\hbox{and}\qquad \limsup_{t\to\infty} \|u(t)\|_\infty=\infty.
\ee
However the blow-up case is in a sense even more surprising, in view of the observation after \eqref{questionUmonot}, and the proof is more delicate than in the global existence case. 
\qed \end{remark}

\smallskip

We do not know presently if strongly nonmonotone behaviors similar to  Theorem~\ref{thmNonMon} or to \eqref{supinf} can occur in bounded domains.
However, more can be said on the related 
question:
\be{questionP}
\hbox{Does $U(t)$ become time nondecreasing whenever it is large enough ?}
\ee
We will see in the next section that this question is strongly connected with certain gradient estimates from the classical  paper \cite{FML}.
To give a more precise formulation of \eqref{questionP},
the question is whether, for given $A>0$, there exists a constant $K=K(A,f,\Omega)>0$ such that, for any $t_0\in(0,T)$,
$$\bigl(\|u_0\|_\infty\le A \mbox{ and } U(t_0)>K\bigr)
  \Longrightarrow \hbox{$U$ is nondecreasing on $[t_0,T)$.}\leqno(\hbox{P})$$
Theorem~\ref{thmNonMon} and Remark~\ref{remGlobal} imply in particular 
that property (P) may fail in $\R^n$ for both blow-up and global solutions.
Our next result shows that if may fail also in bounded domains for global solutions.

\begin{theorem} \label{thm1}
Let $n\ge 3$, $\psi\in L^\infty(\Omega)$, $\psi\ge 0$, $\psi\not\equiv 0$ and $f(u)=u^p$.
Assume either
$$\begin{cases}
&\hbox{$p>p_S$ and $\Omega$ convex and bounded; or}\\
\noalign{\vskip1mm}
&\hbox{$p=p_S$, $\Omega=B_R$, $R>0$ and $\psi$ radially symmetric nonincreasing in $|x|$.}
\end{cases}$$
There exists $\lambda^*>0$ with the following property.
For every $K>0$, there exist $\lambda\in(0,\lambda^*)$ and $t_0>0$
such that the solution of \eqref{eqE1} with initial data $u_0=\lambda\psi$ is global and satisfies
$$\|u_\lambda(t_0)\|_\infty>K\quad\hbox{and}\quad \lim_{t\to\infty}\|u_\lambda(t)\|_\infty=0.$$
In particular, $\|u_\lambda(t)\|_\infty$ is not nondecreasing on $[t_0,\infty)$.
\end{theorem}

\medskip

The negative answer to \eqref{questionUmonot} and the failure of (P) 
guaranteed by Theorems~\ref{thmNonMon}-\ref{thm1}
involve the nonlinearity $f(u)=u^p$ with $p\ge p_S$. The following theorem 
shows that the assumption $p\ge p_S$ is actually necessary. 

\begin{theorem} \label{thm1a}
Let $f(u)=u^p$ with $p\in(1,p_S)$. 
Then property (P) is true.
\end{theorem}

Let us emphasize that Theorem~\ref{thm1a} is true for any
smooth domain $\Omega$ (bounded or unbounded)
and $T\le\infty$.
In the special case when $\Omega$ is bounded and convex, or $\Omega=\R^n$ and $u_0\in H^1(\R^n)$,
a positive answer to \eqref{questionUmonot} follows also from \cite[Theorem~1.7]{MZ}
(cf.~also \cite[Theorem~3 and Corollary~2(ii)]{MZ00}).
A more general and stronger version of Theorem~\ref{thm1a} will be given in 
Theorem~\ref{thm1aN} below.

\subsection{Discussion and remarks.}

\begin{remark} \label{rem1} \rm 
(i) {\bf (Positive results for blow-up solutions)} 
The answer to question  \eqref{questionUmonot}  is positive in the following situations:

\smallskip

\begin{itemize}
\item[(a)] $u_0\in C^2(\Omega)\cap BC(\overline\Omega)$, $\Delta u_0+f(u_0)\ge 0$ 
 and $u_0=0$ on $\partial\Omega$.
 Indeed this ensures $u_t\ge 0$ in 
$\Omega\times(0,T)$ by the maximum principle
(see, e.g.,~\cite[Proposition 52.19]{QS});

\smallskip

\item[(b)] $f(u)=u^p$ with $p\in(1,p_S)$ (Theorem~\ref{thm1a}) 
or, more generally, $\lim_{s\to\infty} s^{-p} f(s)=\ell>0$ and $p\in(1,p_S)$ (Theorem~\ref{thm1aN});
\smallskip

\end{itemize}

\noindent as well as in some radial cases, like:
\begin{itemize}
\item[(c)] $\Omega=B_R$, $u_0$ is radially symmetric and nonincreasing\footnote{or more 
generally $u_0$ is radially symmetric and takes its maximum at $x=0$ for $t$ close to $T$} in $r=|x|$
(see \cite{GP86,NS85} and cf.~also \cite[Lemma~23.12]{QS});

\smallskip

\item[(d)] $\Omega=B_R$, $n=2$, $u_0$ is radially symmetric, $f(u)=e^u$
(see \cite{GL}).
\end{itemize}
	
\medskip

(ii) {\bf (Positive results for global solutions)} 
Assume $\Omega$ is bounded. For a large class of superlinear $f$ with subcritical growth (see \cite{Q03}), 
global solutions satisfy a priori estimate of the form 
\be{ae}
\|u(t)\|_\infty\leq C=C(\|u_0\|_\infty,f,\Omega),\qquad t\geq0.
\ee
Hence if we choose $K:=C(A,f,\Omega)$ in (P), then (P) is trivially true. 
In the case $n=1$, this remains true for all $f\in C^1$ which are superlinear in the following sense:
$$\hbox{$F(u):=\int_0^uf(s)\,ds\ge C_1u^\mu-C_2$ and $f(u)u\ge \mu F(u)-C_3$,}$$ 
where $\mu>2$ and $C_1,C_2,C_3>0$
(see \cite[Theorem 6.1]{Q03}).
If $\Omega=\R^n$, and $f$ is a perturbation of the nonlinearity $u^p$ with $1<p<p_S$, then  
(P) again remains true for global solutions due to the a priori estimates of the form \eqref{ae}, see \cite{Q21,QS25}. 
On the other hand, if $n=2$, then the existence of (superlinear) nonlinearities $f$ such that property (P) fails
seems to be open.
\qed \end{remark}

\begin{remark} \label{rem2} \rm 
{\bf (Weaker properties for global solutions)} 
(i) When (P) fails for global solutions, one can consider also the weaker version (Pw) of property (P) 
for $T=\infty$ where 
$K=K(u_0,f,\Omega)$ (i.e. $K$ might depend on $u_0$ and not just on $\|u_0\|_\infty$).
Property (Pw) fails as well when $\Omega=\R^n$ in view of Remark~\ref{remGlobal}. 
On the other hand, (Pw) is trivially true whenever boundedness of global solutions holds,
which is true for instance when $\Omega$ is bounded and convex and $f(u)=u^p$ with $p>p_S$
(see \cite{CDZ07, BS15, So17}).
The latter as well as (Pw) are open questions when $\Omega$ is not convex.

\smallskip

(ii) One could attempt at weakening (P) by restricting $u_0$ 
to (the positive cone of) a linear subspace $X$ of $L^\infty(\Omega)$ and replacing $\|u_0\|_\infty$
with $\|u_0\|_X$.
However, Theorem~\ref{thm1} shows that even this weaker version fails.
\qed \end{remark}

\begin{remark} \label{rem4} \rm 
{\bf (Monotone quantities)}
Assume $f(u)=u^p$ with $p>1$.
Unlike $U(t)$, some other quantities related to the solution $u$ always become time increasing whenever they are large enough.
This is the case for
$\phi(t)=\int_\Omega u(t)\varphi \,dx$
where $\varphi$ is, respectively, the first Dirichlet eigenfunction of the Laplacian if $\Omega$ is bounded,
and any Gaussian $e^{-k|x-x_0|^2}$ ($k>0$, $x_0\in\R^n$) if $\Omega=\R^n$.
This is also the case for $\psi(t)=\|u(t)\|_2$ if $\Omega$ is bounded,
and these facts remain valid for more general nonlinearities;
see, e.g.,~\cite[(17.3) and 17.9)]{QS}.
However these quantities are much weaker than $U(t)$ and need not blow up as $t\to T$ in general
(see, e.g.,~\cite[Remark~17.7(i)]{QS}).
\qed \end{remark}

\begin{remark} \label{rem3} \rm 
{\bf (Failure of property (P))} 
(i) Since the proof of Theorem~\ref{thmNonMon} is quite involved, 
it is worth noting the following easier argument for the failure of the (stronger) property (P) 
when $\Omega=\R^n$, $f(u)=u^p$, $p>p_S$, and $T<\infty$.
Consider
$u_0$ yielding the behavior in \eqref{supinf}. 
Set $A:=\|u_0\|_\infty+1$ and choose $K>0$.
Then there exist $t_2>t_1>0$ such that $\|u(\cdot,t_1)\|_\infty>K>\|u(\cdot,t_2)\|_\infty$.
The continuous dependence on initial data in $L^\infty$ shows the existence of $\eps\in(0,1)$ such that the solution $\tilde u$ 
with initial data $\tilde u_0:=u_0+\eps$ stays bounded on $[0,t_2]$ and
$\|\tilde u(\cdot,t_1)\|_\infty>K>\|\tilde u(\cdot,t_2)\|_\infty$.
In addition, since $\tilde u_0\ge\eps$, $\tilde u$ has to blow up at some $T\in(t_2,\infty)$.

\smallskip
(ii)
Let $\Omega$ and $f$ be as in (i).
The maximal existence time $T(\tilde u_0)$ of the solution $\tilde u$ in~(i)
tends to infinity as $K\to\infty$ (since $T(\tilde u_0)>t_1$ and $t_1=t_1(K)\to\infty$).
The following argument shows the failure of (P) even if we consider only solutions with $T\le C<\infty$.
\cite[Theorem~1.1]{MV07} guarantees the existence of $0<T_1<T_2<\infty$ and $u_0\in L^\infty(\R^n)$, $u_0\ge0$,
such that the proper solution $u$ with initial data $u_0$ (i.e.~the limit of classical solutions $u_k$
with $f(u)$ replaced by $f_k(u):=\min(u^p,k^p)$) is a classical solution in $(0,T_2)\setminus\{T_1\}$,
it blows up incompletely at $T_1$ and completely at $T_2$.
Fix $\eps>0$, $t_0\in(T_1,T_2)$ and choose $K>\|u(\cdot,t_0)\|_\infty$.
It is known that if $\lambda\in(0,1)$, then the solution $u_\lambda$ with initial data $\lambda u_0$
satisfies $T(\lambda u_0)\ge T_2$.
In addition, if $\lambda$ is close to $1$, then $\|u_\lambda(\cdot,T_1)\|_\infty>K>\|u_\lambda(\cdot,t_0)\|_\infty$
and \cite[Theorem~3.1]{MV07} guarantees $T(\lambda u_0)<T_2+\eps$.
Consequently, (P) fails with $T<T_2+\eps$.  

\smallskip

(iii) Let $f(u)=u^{p_S}$. If $n\ne4$, then \cite{AdP25} and \cite{CdPM20} guarantee the existence
of a large class of bounded (non-radial) domains and initial data $\psi$,
such that the corresponding solution is global and unbounded.
The proof of Theorem~\ref{thm1} and \cite[Theorem 1]{QS25a}
then imply that the conclusion in~Theorem~\ref{thm1} is true with $\lambda^*=1$.
\qed \end{remark}

\subsection{Sketch of proof of Theorem~\ref{thmNonMon}.}
 The basic idea is to construct the initial 
data $u_0\in L^\infty(\R^n)$ as the sum of a sequence of functions $v_{0,k}$
obtained by 
choosing a suitable
solution $u^*$ with {\it incomplete} $L^\infty$ blow-up,
suitable $x_k\in\R^n$ and $\tau_k>0$, and multiplying the functions
$u^*(\cdot-x_k,\tau_k)$
by appropriate compactly supported cut-offs. 
The outline (in very simplified form) is as follows.

\vskip 1pt

$\bullet${\hskip 1mm}
For each of the solutions $v_k$ with initial data $v_{0,k}$, 
its norm $\|v_k(\cdot,t)\|_\infty$ has a maximum at a time $t=T_k$ and then 
drops to lower values for $t>T_k$.
The parameters of the above 
 transformations are adjusted in such a way that the sequence $T_k$ increases to a finite limit $T$,
while $\|v_k(\cdot,T_k)\|_\infty$ increases monotonically to~$\infty$, and $t\mapsto \sup_k\|v_k(\cdot,t)\|_\infty$
then describes a curve with oscillatory blow-up as $t\to T$.

\vskip 1pt

$\bullet${\hskip 1mm}From this, the goal is to ensure that the solution $u$ arising from $u_0$ inherits the required oscillation properties,
by showing that the $u-v_k$ remain small on suitable sets~$D_k$
(these sets contain the supports of the $v_{0,k}$ and form a partition of $\R^n$).

\vskip 1pt

$\bullet${\hskip 1mm}To achieve this, one chooses the space shifts in such a way that the centers of $D_k$ move rapidly enough to infinity as $k\to\infty$ (which will in particular produce blow-up at infinity),
and one then shows that the functions $u-v_k$ cannot escape a carefully designed trapping region.
This requires a delicate splitting of the space-time integrals
in the Duhamel representation formula for $u-v_k$, 
making use of the spatial decay properties of $u^*$ and of the Gaussian decay of the heat kernel,
with special care needed to estimate the critical mutual contributions of the adjacent regions $D_k$ and $D_{k+1}$.

\section{Relations between time monotonicity and gradient estimates}

We now introduce the notation
$$M(t_0)=\sup_{t\in[0,t_0]} U(t),\quad F(z)=\int_0^z f(s)ds$$
and the (positive cone of the) space 
$$Y:=\{v\in C^1(\overline\Omega)\,:\,
\hbox{$v\geq 0$ and $v=0$ on $\partial\Omega$}\},\quad \|v\|_Y=\|v\|_\infty+\|\nabla v\|_\infty.$$
We have the following gradient estimate for solutions of \eqref{eqE1}.

\begin{proposition} 
 \label{thm2}
Assume $\Omega$ convex, \eqref{hypf} is satisfied and
$u_0\in Y$.
Let $t_0\in(0,T)$.
If
\be{hypFML}
\frac12|\nabla u_0|^2+F(u_0)\le F(M(t_0))\quad\hbox{ in $\overline\Omega$,}
\ee
then
\be{conclFML}
\frac12 |\nabla u|^2+F(u)\le F(M(t_0))\quad\hbox{ in $\overline\Omega\times[0,t_0]$.}
\ee
\end{proposition} 

\goodbreak

Proposition~\ref{thm2} follows from arguments in \cite{Sp};
see also \cite{FML}  (whose authors seemed unaware of \cite{Sp})
and the proof of \cite[Proposition 24.4a]{QS}.
Since these works assume $\Omega$ bounded,
we will give a proof below with the necessary changes to cover also the case $\Omega$ unbounded. 

On the other hand, it was claimed in \cite[Theorem~3.1]{FML} 
that Proposition~\ref{thm2}
would hold with $U(t_0)$ instead of $M(t_0)$ in \eqref{hypFML},~\eqref{conclFML} i.e., that
\be{hypFML0}
\frac12|\nabla u_0|^2+F(u_0)\le F(U(t_0))\quad\hbox{ in $\overline\Omega$}
\ee
would imply 
\be{conclFML4a}
\frac12 |\nabla u|^2+F(u)\le F(U(t_0))\quad\hbox{ in $\overline\Omega\times[0,t_0]$.}
\ee
However, it was observed (see 
\cite[Proposition~24.4a]{QS} and the subsequent remark)
that the proof of \cite[Theorem~3.1]{FML} 
contains a gap (see Remark~\ref{gapFML} for more details).

Still, it remained somewhat unclear whether the original statement of \cite[Theorem~3.1]{FML} 
might nevertheless be true.
Theorem~\ref{thm1} combined with the following simple proposition 
entail that this is not the case: $M(t_0)=U(t_0)$ is a necessary (and sufficient) condition for 
\eqref{conclFML4a} to hold,
whereas Theorem~\ref{thm1} shows that the two quantities may differ 
while assumption \eqref{hypFML0} is satisfied.

\begin{proposition} \label{prop3}
Assume \eqref{hypf} and let $0\le t_1<t_0<T$.

\smallskip

(i)  If 
\be{hypFML3}
\frac12|\nabla u|^2+F(u)\le F(U(t_0))\quad\hbox{ in $\overline\Omega\times\{t_1\}$,}
\ee
then $U(t_1)\le U(t_0)$.
\smallskip

(ii) Assume $\Omega$ convex, $u_0\in Y$ and \eqref{hypFML0}.
Property \eqref{conclFML4a} is satisfied if and only if $M(t_0)=U(t_0)$.
\end{proposition}

In one dimensional or radial situations, we even have a
necessary condition, for the validity of the gradient estimate at a given time.

\begin{proposition} \label{prop4}
Assume \eqref{hypf} and either $n=1$ and $\Omega$ is bounded, or $\Omega=B_R$ and $u_0$ radially symmetric.

\smallskip

(i)  The left and right derivatives $U'_-(t)$ and $U'_+(t)$ of the function $U(t)$ exist for all $t\in(0,T)$,
and $U'_+(t)\ge U'_-(t)$. Moreover $U(t)$ is differentiable except for at most countably many times.

\smallskip

(ii) Let $t_0\in(0,T)$ and assume that either $n=1$ or $U(t_0)>u(0,t_0)$. If 
\be{conclFML4}
\frac12|\nabla u|^2+F(u)\le F(U(t_0))\quad\hbox{ in $\overline\Omega\times\{t_0\}$,}
\ee
then $U'_-(t_0)\ge 0$. 
\end{proposition}

Estimate \eqref{conclFML4} near the maximum points of $u(\cdot,t_0)$
is applied in \cite{FML} and in subsequent works \cite{lacey, fila2008nonconstant}
to derive various properties of blow-up solutions
as $t\to T<\infty$, where $T$ is the blow-up time. 
Its applicability thus depends on the availability of the property:
$$\hbox{$U(t)$ is nondecreasing for $t$ close to $T$,}\leqno(P')$$
 but Theorem~\ref{thmNonMon} shows that the latter may fail, at least for $\Omega=\R^n$.
Without knowing (P'), one can a priori only assert the existence of a {\it sequence} of times $t_j\to T$
such that $U(t_j)=M(t_j)$, so that estimate \eqref{conclFML4} holds at times $t_j$ by Proposition~\ref{prop3}(ii).
Property (P') is however known to hold under various sufficient conditions (see Remark~\ref{rem1}(i)). 
\smallskip

Regarding the applications of estimate \eqref{conclFML4} in \cite{FML,lacey, fila2008nonconstant},
the situation is as follows:

\smallskip
\begin{itemize}

\item[$\bullet$] In \cite[Theorem~3.2]{FML}, for problem \eqref{eqE1} with $f(s)=(s+\lambda)^p$, $p>1$, $\lambda\ge 0$,
property \eqref{conclFML4} is used to show the blow-up of supercritical norms, i.e.:
\be{BUnorm}
\lim_{t\to T}\|u(t)\|_q=\infty,\quad q>n(p-1)/2.
\ee
Without knowing property (P'), the argument in the proof of \cite[Theorem~3.2]{FML}
only yields \eqref{BUnorm} with a $\limsup$ instead of a limit.
However, property \eqref{BUnorm} (even in nonconvex domains) follows from \cite{We},
which is based on, completely different, semigroup techniques.
See Remark~\ref{remFML} for further comments.
 
 \smallskip
 
\item[$\bullet$] In \cite[Theorem~1]{fila2008nonconstant}, 
 a type I blow-up estimate is proved for problem \eqref{eqE1} with $f(s)\sim e^s$ as $s\to\infty$,
 and property \eqref{conclFML4} is used in the proof.
The conclusion is valid since case (c) from Remark~\ref{rem1}(i) occurs under their assumptions.

\smallskip

\item[$\bullet$] In \cite{lacey}, for problem \eqref{eqE1} with $f(s)\sim s\log^ps$, as $s\to\infty$, and $p\in(1,2]$,
property \eqref{conclFML4} is used to show global or regional blow-up of all solutions, namely:
\be{BUglobal}
\lim_{t\to T}u(x,t)=\infty,\quad
	\begin{cases} 
	&\hbox{for all $x\in\Omega$, if $p\in(1,2)$} \\
	&\hbox{for all $x\in\Sigma$, if $p=2$,}
		\end{cases} 
\ee
where $\Sigma$ is some subset of positive measure.
Since property \eqref{conclFML4} is used at all times close to $T$,
the arguments in \cite{lacey} do not seem to work without knowing property (P')
(the latter holds for instance in case (a) or (c)  from Remark~\ref{rem1}(i)).
However, property \eqref{BUglobal} was proved recently in \cite{LBS} 
(even in nonconvex domains and for more general problems),
 by different arguments which avoid the use of \eqref{conclFML4}.
For radial solutions, results on global or regional blow-up can be found in \cite{GV1, GV2}.
\end{itemize}

\begin{remark} \label{remFML} \rm
The proof of \cite[Theorem~3.2]{FML} is based on the estimate 
\be{r0}
r_0>c/(u_0^{(p-1)/2}), 
\ee
where $u_0=u(x_0,t_0)=U(t_0)$, $c=1/4$
and $r_0=\sup\{r:u(x,t_0)\ge u_0/2\hbox{ if }|x-x_0|\le r\}$.
Fix $\alpha\in(0,1)$.
If $U(t_0)\ge \alpha M(t_0)$, where $M(t_0)=\sup_{t\le t_0}U(t)$,
then the arguments in \cite{FML} can be modified to prove \eqref{r0}
with $c=\alpha^{(p+1)/2}/4$, which would be sufficient for the conclusion of \cite[Theorem~3.2]{FML}.
However, Theorem~\ref{thmNonMon} shows that the inequality $U(t_0)\ge \alpha M(t_0)$
need not be true in any left neighborhood of $T$.

Similar comment applies to the arguments in \cite{lacey}.
\qed \end{remark}

The rest of the paper is organized as follows.
Theorem~\ref{thmNonMon} is proved in Section~\ref{SecProofs}.
In Section~\ref{SecProof2} we prove Theorems~\ref{thm1}-\ref{thm1a}
and state and prove Theorem~\ref{thm1aN} (a more general and stronger version of Theorem~\ref{thm1a}).
Finally, Section~\ref{SecProof3} is devoted to the proof of the results concerning gradient estimates, namely Propositions~\ref{thm2}-\ref{prop4}.


\section{Proof of Theorem~\ref{thmNonMon}}
\label{SecProofs}

We begin this section by stating the precise local existence-uniqueness result mentioned before property \eqref{BUalt}
		(see, e.g.,~\cite[Definition 15.1 and Proposition 16.1]{QS}):
There exist a unique $T=T(u_0)\in(0,\infty]$  and a unique function $u\ge 0$ such that
		$$u\in C^{2,1}(\overline\Omega\times(0,T))\cap C((0,T),L^\infty(\Omega)),
\quad\lim_{t\to0}\|u(t)-e^{-tA}u_0\|_\infty=0,$$
where $e^{-tA}$ is the Dirichlet heat semigroup in $\Omega$,
$u$ is a classical solution of \eqref{eqE1} for $t\in(0,T)$,
and \eqref{BUalt} holds.

\begin{proof}[Proof of Theorem~\ref{thmNonMon}]
We shall actually prove 
the existence of an increasing sequence $N_k\nearrow\infty$ such that,
for any increasing sequence $\{M_k\}$ satisfying $M_k>N_k+1$, 
there exist $u_0\in L^\infty(\R^n)$, $u_0\ge 0$, and
an increasing sequence $t_k\nearrow T$, such that
\be{nonmonotBU2}
\|u(\cdot,t_{2k})\|_\infty<N_k+1<M_k<\|u(\cdot,t_{2k-1})\|_\infty<M_k+1.
\ee
This will immediately imply Theorem~\ref{thmNonMon} (by considering the sequence $N'_k:=N_k+1$).
Since the proof  is rather involved, we split it into several steps for clarity.
 Throughout the proof, $C$ denotes a generic positive constant depending only on~$n,p$.

\smallskip

{\bf Step 1.} {\it Choice of a threshold solution.}
Recall the Joseph-Lundgren and Lepin exponents $p_{JL}, p_L$, respectively given by \eqref{pJL} and 
\be{pL}  
  p_L:=\begin{cases}
         \infty & \hbox{if }n\leq10, \\
         1+\frac6{n-10} & \hbox{if }n>10,
        \end{cases}
\ee
and note that $p_{JL}<p_L$ if $n>10$. 

First assume $p<p_L$. Fix $T^*>0$, $\delta\in(0,T^*]$ and let
\be{ustarpL}
u^*(x,t):=\begin{cases}
\frac1{(T^*-t)^\beta}h(\frac{|x|}{\sqrt{T^*-t}}), & t<T^*, \\
\frac1{(t-T^*)^\beta}g(\frac{|x|}{\sqrt{t-T^*}}), & t>T^*,
\end{cases} 
\ee
be the peaking solution with $h,g\in C^2([0,\infty))$ being positive solutions
of the equations
$$ \begin{aligned}
h''+\Bigl(\frac{n-1}r-\frac r2\Bigr)h'-\frac1{p-1}h+h^p &=0, \quad r>0, \\
g''+\Bigl(\frac{n-1}r+\frac r2\Bigr)g'+\frac1{p-1}g+g^p &=0, \quad r>0, 
\end{aligned}
$$
satisfying conditions $h'(0)=g'(0)=0$ and
$\lim_{r\to\infty}r^{2\beta}h(r)=\lim_{r\to\infty}r^{2\beta}g(r)=\ell>0$,
see \cite{GV97} and \cite[Appendix Ga]{QS}.
Notice that $u^*(x,T^*):=\lim_{t\to T^*}u^*(x,t)=\ell|x|^{-2\beta}$ for
$x\ne0$.
The function $u^*(\cdot,t)$ is radially nonincreasing 
(see \cite[Lemma 2.2]{M09} if $t<T^*$) and $u^*(0,t)$ 
is increasing from $0$ to $\infty$ for $t\in(-\infty,T^*)$
and decreasing from $\infty$ to $0$ for $t\in(T^*,\infty)$.

Next assume $p\ge p_L$ ($p_S<p\ne p_{JL}$ would do). 
Let $\psi\in{\mathcal D}(\R^n)$ be nonnegative, radially symmetric and nonincreasing,
$\psi\not\equiv0$.
Let $u_\lambda$ be the solution with initial data $\lambda\psi$,
$$\lambda^*:=\sup\{\lambda>0:\hbox{$u_\lambda$ is global and $\|u_\lambda(\cdot,t)\|_\infty\to0$ as $t\to\infty$}\}$$
and let $u^*:=u_{\lambda^*}$ be the threshold solution.
It is known that $u^*$ blows up in $L^\infty(\R^n)$ at some $T^*\in(0,\infty)$,
the blow-up is incomplete
(and of type II if $p>p_L$, see \cite[Proposition~1.8]{MM11}).
Moreover,
the proper continuation (still denoted by $u^*$) is smooth 
in $Q:=(\R^n\times(0,T^*+ \delta])\setminus\{(0,T^*)\}$
for some $\delta>0$  and we have $u_\lambda\nearrow u^*$ in $Q$, see  
\cite[property (iv) on p.~720 and Section~2.4]{MM11}  and cf.~also
\cite[Corollary 5.20, Theorem 5.21]{MM09}, \cite[Lemma 2.5]{MM11} and \cite[Section~2]{GV97}.
We can assume $\delta\le T^*$.

We also claim that
\be{Nboundfg}
u^*(x,t)\leq C|x|^{-2/(p-1)},\qquad x\ne0,\quad 0<t\le T^*+\delta.
\ee
In case $p<p_L$, this is a direct consquence of \eqref{ustarpL}.
If $p\ge p_L$, then this follows from
 the Kaplan blow-up argument in $B_R$ (see \cite[Lemma~2.1]{M02}), 
used for the global radially nonincreasing solutions $u_\lambda$.

\smallskip

{\bf Step 2.} {\it Construction of the solution.} 
Since the function $[0,T^*)\cup(T^*,T^*+\delta]\to(0,\infty):t\mapsto  u^*(0,t)=\|u^*(\cdot,t)\|_\infty$ is continuous
and $\lim_{t\to T^*}u^*(0,t)=\infty$, given $N\gg1$, there exist
$\delta/2<t_1(N)<T^*<t_2(N)<T^*+\delta/2$ such that 
\be{NtM}
\max_{t\in[0,t_1(N)]}u^*(0,t)=u^*(0,t_1(N))=u^*(0,t_2(N))=\max_{t\in[t_2(N),t^*+\delta]}u^*(0,t)= N.
\ee
In addition 
 $\lim_{N\to\infty} t_1(N)=\lim_{N\to\infty} t_2(N)=T^*$.
Choose an increasing sequence $N_k\nearrow\infty$ such that 
\be{N1}
N_1>\sup_{t\in[0,\delta/2]\cup[T^*+\delta/2,T^*+\delta]}u^*(0,t)+1
\ee 
and 
\be{Deltak}
\Delta_k:=t_{k,2}-t_{k,1}\le 2^{-k-2}\delta,\quad\hbox{where}\quad t_{k,i}:=t_i(N_k),\quad k=1,2,\dots.
\ee
Set $T_1:=T^*-\delta/2$ and $T_{k+1}:=T_k+\Delta_k$.
 Then we have 
\be{defT}
T:=\lim_{k\to\infty}T_k\in(T^*-\delta/2,T^*-\delta/4]
\qquad\hbox{and}\qquad
 \lim_{k\to\infty}t_{k,i}=T^*,\ i=1,2.
\ee

Let $M_k$  be any increasing sequence satisfying $M_k>N_k+1$. 
Let $\varphi\in C^\infty(\R^n)$ be radial, radially nonincreasing and 
satisfy $\varphi(x)=1$ for $|x|\le 1$, $\varphi(x)=0$ for $|x|\ge 2$.
By continuous dependence with respect to initial data (in $L^q(\R^n)$ with $q>n(p-1)/2$) 
we may select  $\rho_k>0$ such that
the radially nonincreasing solution $w_k$ 
 with initial data  
$$w_k(x,0)=u^*(x,T^*-T_k)\varphi(\rho_k^{-1}x)$$ 
exists on $[0,T_k+\delta]$ and
satisfies
\be{wkMk}
\sup_{\R^n\times[0,T_k+\delta]} w_k=w_k(0,t_k)=M_k,
\ee 
with $T_k+\delta\ge T^*+\delta/2$.
Notice that 
\be{wkustar}
w_k(x,t)\le u^*(x,t+T^*-T_k),
\ee
 hence 
\be{tkTstar}
t_k\in(t_{k,1}+T_k-T^*,t_{k,2}+T_k-T^*)
\ee
 and
\be{NsupAk}
A:=\sup_{k\ge 1}\|w_k(\cdot,0)\|_{L^\infty(\R^n)}
    =\sup_{k\ge 1}w_k(0,0) \le\sup_{t\in[0,\delta/2]}u^*(0,t)<\infty.
\ee

\begin{figure}[ht]
\centering
\begin{picture}(420,180)(0,0)
\unitlength=0.7pt
\put(286,-16){\makebox(0,0)[c]{$t_k+T^*-T_k$}} 
\put(288,-7){\vector(1,3){9}}
\put(307,6){\makebox(0,0)[c]{$T^*$}}
\put(495,6){\makebox(0,0)[c]{$t$}}
\put(215,6){\makebox(0,0)[c]{$t_{k,1}$}}
\put(333,6){\makebox(0,0)[c]{$t_{k,2}$}}
\put(90,6){\makebox(0,0)[c]{$T^*-T_k$}}
\put(150,130){\makebox(0,0)[c]{$\fblue u^*(0,t)$}}
\put(375,140){\makebox(0,0)[c]{$\fblue u^*(0,t)$}}
\put(460,95){\makebox(0,0)[c]{$\fred w_k(0,t-T^*+T_k)$}}
\put(399,87){\vector(-1,-3){6}}
\put(293,220){\line(1,0){50}}
\put(360,220){\makebox(0,0)[c]{$M_k$}}
\put(20,20){\vector(1,0){480}}
{\linethickness{1pt}
{\fblue%
\bezier500(310,240)(350,50)(500,40) 
\bezier500(20,40)(180,50)(220,188)
\bezier500(220,188)(226,210)(234,210)
\bezier500(234,210)(242,210)(248,188)
\bezier500(248,188)(254,166)(260,166)
\bezier500(260,166)(270,166)(290,240)
}%
{\fred%
\bezier500(80,49)(180,70)(214,150)
\bezier500(214,150)(230,200)(234,200)
\bezier500(234,200)(240,200)(245,175)
\bezier500(245,175)(250,158)(262,158)
\bezier500(262,158)(270,158)(279,185)
\bezier500(279,185)(290,220)(296,220)
\bezier500(296,220)(310,220)(325,170)
\bezier500(325,170)(370,41)(500,33)
}%
}
\put(221,189){\line(1,0){114}}
\put(350,188){\makebox(0,0)[c]{$N_{k}$}}
\newcount\xpos \newcount\ypos \newcount\nr

\xpos=300 \ypos=20
\loop \ifnum\ypos<240
\put(\xpos,\ypos){\line(0,1){3}}
\advance\ypos by6
\repeat

\xpos=296 \ypos=20
\loop \ifnum\ypos<220
\put(\xpos,\ypos){\line(0,1){3}}
\advance\ypos by6
\repeat

\xpos=220 \ypos=20
\loop \ifnum\ypos<187
\put(\xpos,\ypos){\line(0,1){3}}
\advance\ypos by6
\repeat
\put(275,187){\line(0,1){3}}

\xpos=323 \ypos=20
\loop \ifnum\ypos<187
\put(\xpos,\ypos){\line(0,1){3}}
\advance\ypos by6
\repeat
\put(323,187){\line(0,1){3}}

\put(80,20){\line(0,1){3}}
\put(80,26){\line(0,1){3}}
\put(80,32){\line(0,1){3}}
\put(80,38){\line(0,1){3}}
\put(80,44){\line(0,1){3}}

\end{picture}
\vskip5mm
\caption{Graphs of $\fblue t\mapsto u^*(0,t)$
         and $\fred t\mapsto w_k(0,t-T^*+T_k)$.}
   \label{fig-psi2}
\end{figure}
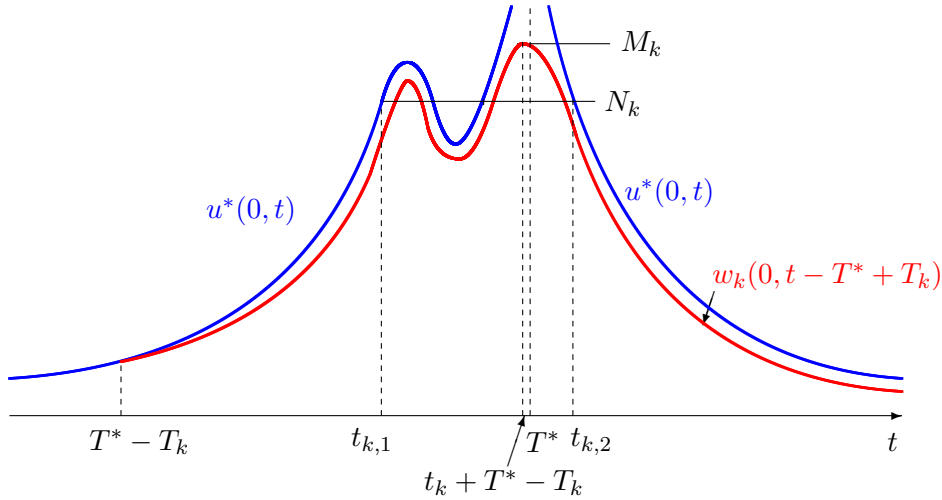


Consider $x_k=(x_k^1,0,0,\dots,0)\in\R^n$
such that $x_1^1=0$
and 
\be{defxk}
x_{k+1}^1=x_k^1+R_k'+R_{k+1}',
\ee
where the numbers $R_1'<R_2'<\dots$ satisfying $R'_k>4\rho_k$ will be chosen in Step~7 below.
Let 
$$v_k(x,t)=w_k(x-x_k,t)$$ 
 be the solution with initial data
$v_{0,k}(x)=w_k(x-x_k,0)$.
Notice that
$\sup_{\R^n\times[0,T_k+\delta]} v_k=v_k(x_k,t_k)=M_k$.
 Set $\tilde t_k:=t_{k,2}+T_k-T^*$. 
 Using \eqref{tkTstar}, \eqref{NtM} and \eqref{Deltak}, we get $t_k<\tilde t_k$ and
\be{tkmonot0}
t_{k+1}>t_{k+1,1}+T_{k+1}-T^*>t_{k,1}+T_{k+1}-T^*=t_{k,2}+T_k-T^*,
\ee
hence
\be{tkmonot}
t_k<\tilde t_k<t_{k+1},
\ee
and we also have
\be{limtk}
\lim_{k\to\infty}t_k=T
\ee
in view of \eqref{defT}, \eqref{tkTstar}.
We claim that
\be{Ntildetk0}
\|v_j(\cdot,\tilde t_k)\|_\infty\le N_k,\quad j, k\ge 1
\ee
and
\be{Ntildetk1}
\|v_j(\cdot,t_k)\|_\infty\le M_k,\quad j, k\ge 1.
\ee
Indeed, since 
$$\tilde t_k+T^*-T_j=t_{k,2}+T_k-T_j
\ \begin{cases}
 \ \le t_{k,1} &\quad\hbox{if }\  j>k,\\
 \ \ge t_{k,2} &\quad\hbox{if }\ j\le k,
\end{cases}
$$
and \eqref{NtM} is true with $N=N_k$, $t_1(N)=t_{k,1}$ and $t_2(N)=t_{k,2}$,
 it follows from \eqref{wkustar} that
$$
\|v_j(\cdot,\tilde t_k)\|_\infty\le u^*(0,\tilde t_k+T^*-T_j)
\le \begin{cases}
u^*(0,t_{k,1})=N_k &\quad\hbox{if }\ j>k,\\
u^*(0,t_{k,2})=N_k &\quad\hbox{if }\ j\le k,
\end{cases}
$$
hence \eqref{Ntildetk0}.
On the other hand, if $j\le k$, then 
$\|v_j(\cdot,t_k)\|_\infty\le M_j\le M_k$ by \eqref{wkMk} whereas, if $j>k$ then
$t_k+T^*-T_j\le t_k+T^*-T_k-\Delta_k<t_{k,2}-\Delta_k=t_{k,1}$ owing to \eqref{tkTstar}
 and \eqref{Deltak}, so that
$\|v_j(\cdot,t_k)\|_\infty\le u^*(0,t_k+T^*-T_j)\le N_k<M_k$ by \eqref{NtM} and \eqref{wkustar}.
This proves~\eqref{Ntildetk1}.

Let $u_m$ ($m\ge 2$) and $u$ be the solutions with initial data 
$$u_m(x,0)=\sum_{k=1}^m v_{0,k}(x),\qquad 
u(x,0)=\sum_{k=1}^\infty v_{0,k}(x),$$
which belong to $L^\infty(\R^n)$, with $\|u_m(\cdot,0)\|_{L^\infty(\R^n)}, 
\|u(\cdot,0)\|_{L^\infty(\R^n)}\le A$,
in view of \eqref{NsupAk} and of the fact that the functions $v_{0,k}$ have disjoint compact supports.
 Denote their maximal existence time by $T^m_*, T_*$, respectively.
For each $m\ge 2$ and $k\in\{1,\dots,m\}$,
 since $v_{0,k}$ and $u_m(\cdot,0)$ are continuous with compact support,
 standard heat semigroup properties imply
\be{contumvk}
 u_m-v_k\in C(\R^n\times[0,T_*^m)),\quad \lim_{R\to \infty}\Big(\sup_{|x|\ge R,\,s\in[0,t]} |(u_m-v_k)(x,s)|\Big)=0,
\ t<T_*^m.
\ee

\smallskip

{\bf Step 3.} {\it Sought-for estimate on $u-v_k$.} 
Fix any $\tilde T\in(T,T^*+\delta/2)$. Denote 
$$D_1':=(-\infty,x_1^1+R_1']\times\R^{n-1},\qquad
D_i':=[x_i^1-R_i',x_i^1+R_i']\times\R^{n-1},\quad i\ge 2.$$
We will show that 
\be{NestDk0}
 T^m_*>\tilde T \quad\hbox{and}\quad
 \sup_{D_k'\times[0,\tilde T]} (u_m-v_k)\le 1/2,\quad m\ge 2,\ k\in\{1,\dots,m\}.
\ee
Since $u_m\nearrow u$ in $\R^n\times(0,T_*)$, inequality 
\eqref{NestDk0} implies
\be{NestDk}
\sup_{D_k'\times[0,\min(\tilde T,T_*))} (u-v_k)\le 1/2,\quad k\ge 1
\end{equation}
and, since 
$\R^n=\bigcup_{k\ge 1}D_k'$ (owing to $x_{k+1}^1=x_k^1+R_k'+R_{k+1}'$), 
it follows from \eqref{NestDk} that
\be{NestDkB}
\|u(\cdot,t)\|_\infty \le \frac 12+\sup_{j\ge 1} \|v_j(\cdot,t)\|_\infty,\quad t\in [0,\min(\tilde T,T_*)).
\ee
Consequently, for any $T'\in(0,\min(T,T_*))$, using \eqref{wkMk}, \eqref{wkustar}, we obtain
$$
\begin{aligned} 
\|u(\cdot,t)\|_\infty &\le \frac 12+\sup_{j\ge 1} \|v_j(\cdot,t)\|_\infty
\le \frac 12+\max\Big(M_{k_0},\sup_{k\ge k_0} u^*(0,t-T_k+T^*)\Big)\\
&\le \frac 12+\max\Big(M_{k_0},\sup_{0\le \tau<T'-T_{k_0}+T^*}u^*(0,\tau)\Big)
<\infty,\quad t\in(0,T'),
\end{aligned}$$
where $k_0=k_0(T')$ is such that $T_{k_0}>T'$.
 We deduce that $T_*\ge T$.

On the other hand, \eqref{NestDkB} combined with \eqref{Ntildetk0}-\eqref{Ntildetk1} 
guarantees that
\be{nonmonot1}
\|u(\cdot,\tilde t_k)\|_\infty<N_k+1,\quad
\|u(\cdot,t_k)\|_\infty<M_k+1,
\ee
whereas \eqref{wkMk} implies
\be{nonmonot2}
\|u(\cdot,t_k)\|_\infty>M_k.
\ee
 This along with \eqref{limtk} guarantees that $u$ blows up in $L^\infty(\R^n)$ at time $T_*=T$.
Relabeling $t_k$ as $t_{2k-1}$ and $\tilde t_k$ as $t_{2k}$,
 it follows from \eqref{tkmonot} that the new sequence $\{t_k\}$ is increasing, and
inequalities \eqref{nonmonot1} and \eqref{nonmonot2} yield the non-monotonicity property \eqref{nonmonotBU2}.

\smallskip

{\bf Step 4.} {\it Trapping region.}
In order to prove \eqref{NestDk0},
we will show that the functions $u_m-v_k$ cannot escape a suitably defined trapping region.
This will require a careful splitting of the space-time integrals
in the Duhamel representation formula \eqref{Nidentvk} for these functions.

Fix $m\ge 2$, 
set $\tilde D_i'=D_i'$ if $1\le i\le m-1$ and $\tilde D_m'=[x_m^1-R_m',\infty)\times\R^{n-1}$.
Set also 
$$D_1:=(-\infty,x_1^1+R_1]\times\R^{n-1},\qquad
D_i:=[x_i^1-R_i,x_i^1+R_i]\times\R^{n-1},\quad i\ge 2,$$
where the numbers
$R_1<R_2<\dots$ satisfying $R_1>2\sqrt{\tilde T}$, $R_i>2\rho_i$ and $R_i'\ge R_i+R_{i+1}$ will be chosen in
Step~7,
$\tilde D_i:=D_i$ if $1\le i\le m-1$ and $\tilde D_m:=[x_m^1-R_m,\infty)\times\R^{n-1}$.
These sets are schematically represented on Figure~\ref{fig-Di} below
for convenience.

We then define ${\hat v}_i=u_m-v_i$,
$$ Y_i(t):=\eps_ie^{L_it}, \ i=1,2,\dots,m, $$
where $\eps_1>\eps_2>\dots>0$ and $0<L_1<L_2<\dots$ 
satisfy 
\be{Leps}
\begin{aligned}
 &L_i\ge 3(8+2p(M_{i+1}+1)^{p-1}),\quad \eps_i\le\frac12e^{-L_i\tilde T},\quad i=1,2,\dots,\\
 &\eps_i\le\eps_{i-1}e^{-(L_i-L_{i-1})\tilde T},\quad  i=2,3,\dots.
\end{aligned}
\ee
Notice that \eqref{Leps} implies, 
for all $t\in[0,\tilde T]$, $Y_i(t)\le1/2$ for $i=1,2\dots$ and
$Y_{i-1}(t)\ge Y_i(t)$ for $i=2,3,\dots$.
In Step~7, 
 the choice of $R_i,R_i'$ will also guarantee that 
\be{Yi1}
Y_{i-1}\frac{R_i}{R_i'}\le \frac12 Y_i,\quad i=2,3,\dots,m.
\ee

To define our trapping region we consider,
for each $t$, 
a continuous, piecewise linear function
$Y(\cdot,t):\R\to(0,1/2]$ such that
\be{partY0}  
Y(x^1,t):=\begin{cases}
 Y_1(t), & x^1\le x_1^1, \\
 Y_i(t), & x^1=x_i^1, \ i=2,3,\dots,m-1, \\
 Y_m(t),  & x^1\ge x_m^1,
\end{cases}
\ee
and 
\be{partY}  
\begin{aligned}
&\hbox{$Y(\cdot,t)$ is linear on } [x_{i-1}^1,x_i^1],\ i=2,3,\dots,m,\ \hbox{with} \\
&\partial_{x^1}Y(x^1,t)=\frac{Y_{i}(t)-Y_{i-1}(t)}{x_i^1-x_{i-1}^1}\ge-\frac{Y_{i-1}(t)}{R_i'+R_{i-1}'},
\quad x^1\in(x_{i-1}^1,x_i^1), \ i=2,3.\dots,m,
\end{aligned}
\ee 
 where the inequality follows from $Y_i(t)>0$ and \eqref{defxk}.
Note that $Y(\cdot,t)$ is nonincreasing and
\be{infY}
\inf_{\R\times[0,\tilde T]} Y=\eps_m>0.
\ee
We will show that
\be{vkY} 
{\hat v}_k(x,t)\le Y(x^1,t)\ \hbox{ for }\ x\in\tilde D_k',\ \ k=1,2,\dots,m,\ \ t\in[0,\min(\tilde T,T_*^m)),
\ee
and $T^m_*>\tilde T$,
which will imply \eqref{NestDk0} due to $Y\le1/2$.

Denote
$$ E:=\bigl\{t\in [0,\min(\tilde T,T_*^m)];\ \hat v_k(x,t')\le Y(x^1,t') \hbox{ for }(x,t')\in\tilde D_k'\times[0,t),\ k=1,2,\dots,m\bigl\}.$$
 Since \eqref{vkY} is satisfied for $t=0$ due to 
${\hat v}_k(x,0)=\sum_{j\ne k,\,j\le m}v_j(x,0)=0$ for $x\in\tilde D_k'$,  
 it follows from \eqref{contumvk}, \eqref{infY} that 
$$ \tau:=\sup E\in(0,\min(\tilde T,T_*^m)].$$
(The inequality $\tau>0$ follows also from the well-posedness of
\eqref{eqE1} with $f(u)=u^p$ in the space of bounded, uniformly continuous
functions equipped with the $L^\infty$-norm, see \cite[Proposition~7.3.1]{L95}, for example.) 
For Steps 5 and 6 below, 
$$ \hbox{we fix $k\in\{1,\dots,m\}$, $x\in \tilde D_k'$, $t\in[0,\tau)$.}$$
Since $\R^n=\bigcup_{1\le i\le m}\tilde D_i'$, we may write
\be{Nidentvk}
{\hat v}_k(x,t)=I(x,t)+\sum_{i=1}^m J_i(x,t),
\ee
where $I, J_i$ are given by
$$I=\int_{\R^n} G(x-y,t) {\hat v}_k(y,0)dy,\quad J_i =\int_0^t\int_{\tilde D_i'} G(x-y,t-s)(u^p_m-v_k^p)(y,s)dyds.$$ 
Taking advantage of the decay property \eqref{Nboundfg}
 and of the Gaussian decay of the heat kernel, we shall estimate $J_i(x,t)$ by time integrals of $Y(x^1,t)$
and/or by constants that will ultimately converge rapidly to $0$ as $k\to\infty$. 
The above choice of $Y$ is designed in a way to estimate the critical contribution, 
to the adjacent region $J_{k-1}$ (and $J_{k+1}$), of the neighboring points to $x$, for which the decay of the heat kernel is not effective.

\smallskip

{\bf Step 5.} {\it Estimate of $I=I(x,t)$.} 
Set 
$$\Sigma_i:=\big(x_i^1-R_i'-R_{i-1},
   x_i^1+R_i'+R_{i+1}\big)\times\R^{n-1},\quad 2\le i\le m-1$$
and $\Sigma_1=(-\infty,x_1^1+R_1'+R_{2})\times\R^{n-1}$, $\Sigma_m=(x_m^1-R_m'-R_{m-1},\infty)\times\R^{n-1}$.
We note that, since $v_j(y,0)=0$ if $|y-x_j^1|\ge R_j$,
we have
${\hat v}_k(y,0)=(u_m-v_k)(y,0)=0$
for all $y\in \Sigma_k$.
Also, since $x\in \tilde D_k'$, 
we have $|x^1-y^1|\ge R_{k-1}$  for all $y\in \R^n\setminus  \Sigma_k$, 
where $R_0:=R_2$.
By \eqref{NsupAk} it follows that
$$I
\le A\int_{\R^{n-1}} (4\pi t)^{-\frac{n-1}{2}}e^{-\frac{|x'-y'|^2}{4 t}}dy'
\int_{|z|\ge R_{k-1}} t^{-\frac12} e^{-\frac{|z|^2}{4t}}dz.$$
Using 
\be{Nintz0}
\int_{|z|\ge  R} t^{-\frac12} e^{-\frac{|z|^2}{4t}}dz
=2\int_{|\xi|\ge  R/2\sqrt{t}} e^{-|\xi|^2}d\xi\le 4\int_{R/2\sqrt{t}}^\infty \xi e^{-\xi^2/2}d\xi=4e^{-R^2/8t}
\ee
for $R\ge2\sqrt{t}$, it follows that
\be{NintI0}
 I \le 4A e^{-\tilde R_{k-1}^2/8\tilde T}.
\ee

\vskip -5mm
\def\ss{\scriptstyle}
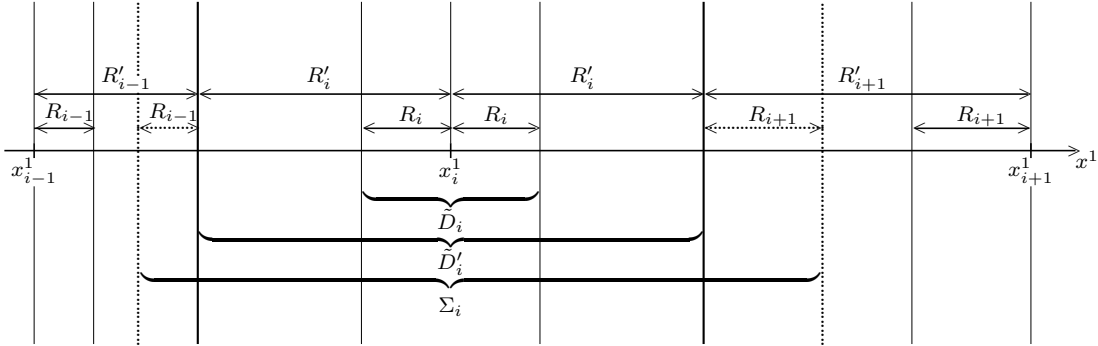
\begin{figure}[ht]
\centering
\begin{picture}(400,220)(45,0) 
\unitlength=0.56pt

\linethickness{0.3pt}
\put(130,70){\line(0,1){230}}
\linethickness{0.8pt}
\put(200,70){\line(0,1){230}}
\put(540,70){\line(0,1){230}}
\linethickness{0.6pt}

\put(90,195){\line(0,1){10}}
\put(90,186){\makebox(0,0)[c]{$\ss x_{i-1}^1$}} 

\put(760,195){\line(0,1){10}}
\put(760,186){\makebox(0,0)[c]{$\ss x_{i+1}^1$}} 

\linethickness{0.3pt}
\put(310,70){\line(0,1){230}}
\put(430,70){\line(0,1){230}}
\put(680,70){\line(0,1){230}}
\put(90,205){\line(0,1){95}}
\multiput(160,70)(0,3){57}{\makebox(0,0){$\ss .$}}
\multiput(160,251)(0,3){17}{\makebox(0,0){$\ss .$}}
\put(370,205){\line(0,1){95}}
\put(90,70){\line(0,1){106}}
\multiput(620,70)(0,3){77}{\makebox(0,0){$\ss .$}}
\put(760,70){\line(0,1){106}}
\put(760,200){\line(0,1){100}}

\linethickness{0.6pt}
\put(70,200){\line(1,0){720}}
\put(784,202){$_{_>}$}
\put(790,188){$\ss x^1$}

\put(90,215){\line(1,0){40}}
\put(90,217){$_{_<}$}
\put(122,217){$_{_>}$}
\put(97,222){$\ss R_{i-1}$}

\multiput(160,215)(3,0){13}{\makebox(0,0){$\ss .$}}
\put(160,217){$_{_<}$}
\put(192,217){$_{_>}$}
\put(167,222){$\ss R_{i-1}$}

\put(90,238){\line(1,0){108}}
\put(90,240){$_{_<}$}
\put(191,240){$_{_>}$}
\put(135,246){$\ss R'_{i-1}$}

\put(370,195){\line(0,1){10}}
\put(370,186){\makebox(0,0)[c]{$\ss x_{i}^1$}} 
\put(370,177){\makebox(0,0)[c]{$\underbrace{\phantom{aaaaaaaaaai)}}$}} 
\put(370,153){\makebox(0,0)[c]{$\ss \tilde D_i$}} 
\put(370,150){\makebox(0,0)[c]{$\underbrace{\phantom{aaaaaaaaaaaaaaaaaaaaaaaaaaaaaaaa)}}$}} 
\put(370,125){\makebox(0,0)[c]{$\ss \tilde D'_i$}} 
\put(390,118){\makebox(0,0)[c]
{$\aunderbrace[l1R]{\phantom{aaaaaaaaaaaaaaaaaaaa}}\hskip -1mm$
$\aunderbrace[L1r]{\phantom{aaaaaaaaaaaaaaaaaaaaaaaa}}$}} 
\put(370,95){\makebox(0,0)[c]{$\ss \Sigma_i$}} 

\put(312,215){\line(1,0){58}}
\put(310,217){$_{_<}$}
\put(362,217){$_{_>}$}
\put(335,220){$\ss R_i$}

\put(372,215){\line(1,0){55}}
\put(370,217){$_{_<}$}
\put(421,217){$_{_>}$}
\put(392,220){$\ss R_i$}

\put(372,238){\line(1,0){167}}
\put(370,240){$_{_<}$}
\put(531,240){$_{_>}$}
\put(450,246){$\ss R'_i$}

\put(202,238){\line(1,0){167}}
\put(200,240){$_{_<}$}
\put(361,240){$_{_>}$}
\put(273,246){$\ss R'_i$}

\put(542,238){\line(1,0){216}}
\put(540,240){$_{_<}$}
\put(751,240){$_{_>}$}
\put(630,246){$\ss R'_{i+1}$}

\multiput(542,215)(3,0){27}{\makebox(0,0){$\ss .$}}
\put(540,217){$_{_<}$}
\put(610,217){$_{_>}$}
\put(570,220){$\ss R_{i+1}$}

\put(682,215){\line(1,0){77}}
\put(680,217){$_{_<}$}
\put(750,217){$_{_>}$}
\put(710,220){$\ss R_{i+1}$}

\end{picture}
\vskip -5mm
\caption{The sets $\tilde D_i$, $\tilde D_i'$ and $\Sigma_i$ for $2\le i\le m-1$.}
   \label{fig-Di}
\end{figure}

{\bf Step 6.} {\it Estimate of $J_i=J_i(x,t)$.}
Let $i\in\{1,\dots,m\}$. 
We consider the following cases separately. Recall that $x\in\tilde D_k'$.
\vskip 1pt

 $\bullet$ {\it Case 1: $i\ge k+2$.} Then, for $y\in\tilde D_i'$ and $s\in[0,t]$, we have $|x^1-y^1|\ge 2R_{i-1}'>4R_{i-1}$
and $u_m(y,s)\le M_i+Y(y^1,s)\le M_i+1$, hence
$$J_i
\le\int_0^t\int_{\R^{n-1}} (4\pi(t-s))^{-\frac{n-1}{2}}e^{-\frac{|x'-y'|^2}{4(t-s)}}dy'
\int_{|z|\ge  4R_{i-1}} (t-s)^{-\frac12} e^{-\frac{|z|^2}{4(t-s)}}(M_i+1)^pdzds.$$
Using \eqref{Nintz0}, 
 it follows that
\be{NintJi1}
J_i\le 4\tilde T e^{-2R_{i-1}^2/\tilde T}(M_i+1)^p,\quad\hbox{if } k+2\le i\le m.
\ee

 $\bullet$ {\it Case 2: $i\le k-2$.} Then $|x^1-y^1|\ge 2R_{k-1}'>4R_{k-1}$ for $y\in \tilde D_i'$, and we similarly obtain
\be{NintJi2}
J_i\le 4\tilde T e^{-2R_{k-1}^2/\tilde T}(M_i+1)^p,\quad\hbox{if } 1\le i\le k-2.
\ee

 $\bullet$ {\it Case 3:
$|i-k|\le1$.} Set
$$(\tilde D_i')^+=\tilde D_i'\cap\{y:|y^1-x^1|>R_k\},\quad (\tilde D_i')^-=\tilde D_i'\cap\{y:|y^1-x^1|\le R_k\}$$
and split $J_i=J_i^++J_i^-$ accordingly. 
The same estimates as above yield
\be{NintJi6}
J_i^+\le 4\tilde T e^{-R_{k}^2/8\tilde T}(M_i+1)^p,\quad\hbox{if } |i-k|\le1.
\ee

 It remains to estimate $J_i^-$.
To this end, setting $\hat D:=\bigcup_{|j-k|\le1}\tilde D_j'$,
we first prove the following inequalities
\be{Yk2}
Y_k(s)\le 2Y(x^1,s),
\ee
\be{Yyx}
Y(y^1,s)\le 2Y(x^1,s),\quad  y\in  \hat D^-:=\hat D\cap\{y:|y^1-x^1|\le R_k\}.
\ee
Since $Y(\cdot,s)$ is nonincreasing, we may assume $x^1\ge x^1_k$ in
the proof of \eqref{Yk2} and $y^1\le x^1$ in the proof of \eqref{Yyx}.
 In addition, since $Y(x^1,s)=Y_k(s)$ if $x^1\ge x^1_k$ and $k=m$,
we may assume $k<m$ in the proof of \eqref{Yk2}.
If $x^1\ge x^1_k$ and $k<m$, then $x^1\in[x^1_k,x^1_k+R_k']$ and the Mean Value
Theorem together with \eqref{partY} and $R_{k+1}'\ge R_k'$ imply
$$ Y(x^1,s)\ge Y_k(s)-R_k'\frac{Y_k(s)}{R_{k+1}'+R_k'}\ge \frac12Y_k(s),$$
hence \eqref{Yk2}. 
Now let 
$y\in\hat D^-$, $y^1\le x^1$, hence $y^1\in[x^1-R_k,x^1]$.  
We have $y^1\ge x^1-R_k\ge x^1_k-R'_k-R_k \ge x^1_k-R'_k-R'_{k-1}= x^1_{k-1}$ if $k>1$.
Notice that \eqref{Yi1}-\eqref{partY} and $R_{k+1}'\ge 2R_k$ imply
$$ |\partial_{x^1}Y(z,s)|\le 
\max\Bigl(\frac{Y_{k-1}(s)}{R_k'},\frac{Y_k(s)}{R_{k+1}'}\Bigr)\le \frac12\frac{Y_k(s)}{R_k},\quad
z\in(x_{k-1}^1,x_{k+1}^1)\cap(x^1_1,x^1_m),$$
with $Y_0=0$ and $x_0^1=-\infty$ if $k=1$, and $\partial_{x^1}Y(z,s)=0$ if $z<x^1_1$ or $z>x^1_m$.
 The Mean Value Theorem and \eqref{Yk2} now imply
$$ Y(y^1,s)\le Y(x^1,s)+R_k\Bigl(\frac12\frac{Y_k(s)}{R_k}\Bigr)\le 2Y(x^1,s),$$
hence \eqref{Yyx}.

\vskip 1pt
 We next estimate $J_i^-$ according to the following two subcases.
\vskip 1pt

 $\bullet$ {\it Case 3.1: $i=k$.} For $y\in (\tilde D_k')^-$, using $Y\le1/2$ and \eqref{Yyx} we obtain
$$(u^p_m-v^p_k)(y,s)\le pu_m^{p-1}{\hat v}_k(y,s)\le p(M_k+1)^{p-1}Y(y^1,s)\le 2p(M_k+1)^{p-1}Y(x^1,s),$$
hence
\be{NintJi5}
J_k^-\le 2p(M_k+1)^{p-1}\int_0^t Y(x^1,s)\,ds.
\ee

$\bullet$ {\it Case 3.2: $|i-k|=1$.
For} $y\in (\tilde D_i')^-$,  
we have  $|y^1-x_i^1|\ge R_i'-R_k\ge R_i$,
hence  $v_i(y,s)\le CR_i^{-2\beta}$  by \eqref{Nboundfg},
so that $Y\le1/2$ and \eqref{Yyx} imply
$$u_m^p(y,s)\le 2^p(v_i^p(y,s)+{\hat v}_i^p(y,s))\le  CR_i^{-2p\beta}+2^pY^p(y^1,s)
\le  CR_i^{-2p\beta}+ 4Y(x^1,s).$$
Therefore,
\be{NintJi3}
J_i^-\le  C\tilde T R_i^{-2\beta p}+ 4\int_0^t Y(x^1,s)\,ds,\quad\hbox{if }|i-k|=1.
\ee

{\bf Step 7. {\it Choice of the sequences $\{R_i\}, \{R'_i\}$ and conclusion.}}
Consider $t<\tau$, $k\in\{1,\dots,m\}$ and $x\in\tilde D_k'$.
Combining \eqref{Nidentvk}, \eqref{NintI0}--\eqref{NintJi6}, \eqref{NintJi5}--\eqref{NintJi3} 
we obtain
$$\hat v_k(x,t)\le \delta_k+(8 +2p(M_k+1)^{p-1})\int_0^t Y(x^1,s)\,ds,$$
where 
$$\delta_k:=  4(A+k\tilde T(M_k+1)^p) e^{-\tilde R_{k-1}^2/8\tilde T}
+ C\tilde T \Big\{R_{k-1}^{-2\beta p}+R_{k+1}^{-2\beta p}+\sum_{i=k+1}^\infty e^{-R_{i-1}^2/8\tilde T}(M_i+1)^p\Big\}.$$
 It is easily seen that the $R_i>2\rho_i$ can be chosen large enough (independently of $m$)
so that 
$\delta_k\le \frac13\eps_{k+1}$.
 Consequently $\delta_k\le\frac13 Y(x^1,t)$,
 owing to $x^1<x_{k+1}^1$.
Once we fixed the $R_i$, we choose $R'_i>4\rho_i$ such that 
the condition $R_i'\ge R_i+R_{i+1}$ mentioned above is satisfied and
$\eps_{i-1}\frac{R_i}{R_i'}\le\frac12\eps_i$, where $\eps_0:=\eps_1$.
This choice guarantees \eqref{Yi1}.
Since $\int_0^t Y(x^1,s)ds\le\frac1{L_{k-1}}Y(x^1,t)$, where $L_0:=L_1$,
\eqref{Leps} implies 
$$ (8+2p(M_k+1)^{p-1})\int_0^t Y(x^1,s)\,ds\le\frac13 Y(x^1,t),$$
hence  $\hat v_k(x,t)\le\frac23 Y(x^1,t)$.
In view of \eqref{contumvk} and \eqref{infY}, it follows that
$\tau=\min(\tilde T,T_*^m)$.
 Recalling $\R^n=\bigcup_{k=1}^m \tilde D_k'$ and \eqref{wkMk}, we deduce
 in particular that 
 $\sup_{t\in [0,\tau)}\|u_m(t)\|_\infty<\infty$.
Consequently, $T_*^m>\tilde T$ and \eqref{vkY} is proved.
This implies \eqref{NestDk0} and
concludes the proof of Theorem~\ref{thmNonMon}.
\end{proof}

\begin{proof}[Proof of Corollary~\ref{corNonMon}]
Let $\kappa=(p-1)^{-1/(p-1)}$ and
choose $M_k=\phi\big(\kappa^{-1}N_k\big)+N_k$ in Theorem~\ref{thmNonMon}.
By the lower blow-up rate estimate (see \cite[Proposition~23.1]{QS}), we have
$$\|u(\cdot,t_{2k})\|_\infty\ge \kappa(T-t_{2k})^{-\beta}.$$
It then follows from \eqref{nonmonotBU} that
$$ \begin{aligned}
\|u(\cdot,t_{2k-1})\|_\infty
&\ge M_k\ge \phi\big(\kappa^{-1}N_k\big)\ge \phi\big(\kappa^{-1}\|u(\cdot,t_{2k})\|_\infty\big)\\
&\ge \phi\big((T-t_{2k})^{-\beta}\big)\ge \phi\big((T-t_{2k-1})^{-\beta}\big).
\end{aligned}$$
\end{proof}

\begin{remark} \label{remTypeI} \rm
(i)
If $p<p_L$, then the sequences $N_k$, $M_k$ in Theorem~\ref{thmNonMon}
can be chosen such that the blow-up is of type I.
In fact, choose 
$$\hbox{$N_1$ satisfying \eqref{N1} 
and $N_1^{p-1}>8(h^{p-1}(0)+g^{p-1}(0))/\delta$, 
$N_{k+1}:=2^\beta N_k$ and $M_k:=2N_k$.}$$
Then, in view of \eqref{ustarpL}, we have
$$t_{k,1}=T^*-(h(0)/N_k)^{p-1}, \ \ t_{k,2}=T^*+(g(0)/N_k)^{p-1},
\ \ \Delta_k=
(h^{p-1}(0)+g^{p-1}(0))N_k^{1-p},$$
hence \eqref{Deltak} is true
and $T-T_k=\sum_{i=k}^\infty\Delta_i\le C/N_k^{p-1}$.
Notice (cf.~\eqref{tkmonot0}-\eqref{limtk}) that $0<t_{k,1}+T_k-T^*\nearrow T$.
If $t\in[t_{k,1}+T_k-T^*,t_{k+1,1}+T_{k+1}-T^*]$,
then
\be{Tt}
T-t\le T-t_{k,1}-T_k+T^*=T-T_k+(h(0)/N_k)^{p-1}\le C/N_k^{p-1}
\ee
and,
arguing as for \eqref{Ntildetk0}, we have
$\sup_x v_j(x,t)\le N_{k+1}$ if $j\ne k$.
Since $v_k\le M_k$, \eqref{NestDk} yields
$\|u(\cdot,t)\|_\infty\le \max(2,2^\beta)N_k+1$,
hence \eqref{Tt} implies that the blow-up is of type I.

\smallskip

(ii) Let $p\ge p_L=p_L(n)$. 
  Then there exists $m\in\{3,\dots,n-1\}$ such that $p_S(m)<p<p_L(m)$
   (one can choose, e.g.,~$m=\lfloor 3+\frac{4}{p-1}\rfloor$).
  If $v:\R^m\times[0,T) \to [0,\infty)$ is a solution with (type I)
  oscillatory blow-up, then $u(x_1,\dots,x_n,t)=v(x_1,\dots,x_m,t)$
  is a solution with (type I) oscillatory blow-up.
\qed \end{remark}  

\begin{remark} \rm
It follows from the proof of Theorem~\ref{thmNonMon} that the solution $u$ can be continued 
as a classical solution for $t\in[T,T')$ for some $T'>T$. However, this extension does not 
satisfy the property $u\in C((0,T'),L^\infty(\R^n))$ since $u(T)\notin L^\infty(\R^n)$.
\qed \end{remark}

\section{Proof of Theorems~\ref{thm1}, \ref{thm1a}}
\label{SecProof2}

\begin{proof}[Proof of Theorem~\ref{thm1}]
Set $\lambda^*:=\sup\{\lambda>0;\ T_\lambda=\infty\}$, where $T_\lambda$
is the maximal existence time of the solution $u_\lambda$. 
It is well known that $\lambda^*\in(0,\infty)$.

\smallskip

First consider the case $p>p_S$ (with $\Omega$ convex bounded). Then $T_{\lambda^*}<\infty$ (see, e.g., \cite[Section~28.4]{QS}), hence 
\be{unbdd0}
\lim_{t\to T_{\lambda^*}}\|u_{\lambda^*}(t)\|_\infty=\infty.
\ee
Moreover, by the definition of $\lambda^*$ and comparison, for all $\lambda\in(0,\lambda^*)$, we have $T_\lambda=\infty$
and we know that $\lim_{t\to T_\lambda}\|u_{\lambda}(t)\|_\infty=0$ (see \cite[Theorem~22.4*]{QS}).
By \eqref{unbdd0}, for any $K>0$, there exists $t_0\in(0,T_{\lambda^*})$ such that $\|u_{\lambda^*}(t_0)\|_\infty>K$.
By continuous dependence, for $\lambda\in(0,\lambda^*)$ close to $\lambda^*$, we have 
$\|u_\lambda(t_0)\|_\infty>K$.

\smallskip

Next consider the case $p=p_S$ (with $\Omega=B_R$ and $\psi$ radially symmetric).
Then the solution $u_{\lambda^*}$ is global unbounded 
(see \cite[Theorem~28.7*(ii)]{QS})
and \cite[Theorem~1]{QS25a} implies $\lim_{t\to\infty}\|u_{\lambda}(t)\|_\infty=0$ for $\lambda<\lambda^*$.
As above, the assertion follows from continuous dependence.
\end{proof}

\begin{proof}[Proof of Theorem~\ref{thm1a}.] 
Assume on the contrary that there exist 
initial data $u_{0,k}$ with $\|u_{0,k}\|_\infty\le A$  
and $t_{2k-1}<t_{2k}$ such that the corresponding solutions $u_k$
are defined on~$[0,t_{2k}]$, 
$U_k(t_{2k-1})>k$,
where $U_k(t):= \|u_k(\cdot,t)\|_\infty$,
and
$$U_k(t_{2k-1})>U_k(t_{2k}) \ \hbox{ for }\ k=1,2,\dots$$ 
If $k>A$, then there exists $\tilde t_k\in(0,t_{2k})$ 
such that $U_k(\tilde t_k)=\max_{[0,t_{2k}]}U_k$.
Choose $x_k\in\Omega$ such that 
$$u_k(x_k,\tilde t_k)>\max\big(U_k(\tilde t_k)-1,U_k(t_{2k})\big)\ge u_k(x_k,t_{2k})$$
and set $\tau_k:=\inf\{\tau\in[\tilde t_k,t_{2k}]:\partial_t u_k(x_k,\tau)\le0\}$.
Then 
$$\partial_tu_k(x_k,\tau_k)\le0\quad \hbox{and}\quad u_k(x_k,\tau_k)\ge u_k(x_k,\tilde t_k)>U_k(\tilde t_k)-1\ge U_k(\tau_k)-1.$$
Set $M_k:=U_k(\tau_k)$, $\lambda_k:=M_k^{(1-p)/2}$ and
$v_k(y,s):=\frac1{M_k}u_k(x_k+\lambda_ky,\tau_k+\lambda_k^2s)$.
Then $v_k$ is a solution of the equation $v_s-\Delta v=v^p$ in $\Omega_k\times(-\tau_k/\lambda_k^2,0]$,
where $\Omega_k:=\{y: x_k+\lambda_ky\in\Omega\}$, $0\le v_k\le1$, $v_k(0,0)>1-1/M_k$, $\partial_sv_k(0,0)\le0$
and $v_k(y,s)=0$ for $y\in\partial\Omega_k$. 
Passing to the limit we obtain an ancient solution $v$ 
of the equation $v_s-\Delta v=v^p$ either in $\R^n\times(-\infty,0]$
or (after changing the coordinate system) in $H^n\times(-\infty,0]$,
where $H^n=\{y\in\R^n:y_1>-c\}$ for some $c>0$,
satisfying the homogeneous Dirichlet boundary condition for $y\in\partial H^n$.
In addition, $0\le v\le1$, $v(0,0)=1$ and $v_s(0,0)\le0$.
In the case of the halfspace $H^n$ we can set $v(-c-y_1,y_2,\dots,y_n,s)=-v(-c+y_1,y_2,\dots,y_n,s)$
for $y_1>0$ to obtain an ancient solution in $\R^n\times(-\infty,0]$.
In both cases, \cite{Q21} guarantees the universal estimate $|v(y,s)|\le C(n,p)(-s)^{-1/(p-1)}$  
and \cite[Corollary~1]{MZ00} implies that $v$ depends only on $s$.
Consequently, $v_s(0,0)=v^p(0,0)=1$, which yields a contradiction.
\end{proof}

With some extra work, we can  actually obtain the following more general and stronger version of Theorem~\ref{thm1a}.

\begin{theorem} \label{thm1aN}
Assume $\lim_{s\to\infty}f(s)s^{-p}=\ell>0$ and $p\in(1,p_S)$.
Then property (P) is true. 
More precisely, for any $\eps\in(0,1)$ 
there exists $K=K(A,f,\Omega,\eps)>0$
such that 
$$\bigl(\|u_0\|_\infty\le A \mbox{ and } U(t_0)>K\bigr)
  \Longrightarrow U'(t)\ge(1-\eps)\ell U^p(t) \hbox{ for a.a. } t\in [t_0,T).$$ 
In particular, if $\|u_0\|_\infty\le A$ and $U(t_0)>K$, then $T(u_0)<\infty$.
\end{theorem}

\begin{proof}[Proof of Theorem~\ref{thm1aN}.]
Rescaling arguments show that we can assume $\ell=1$. 
The function $t\mapsto U(t)$ is locally Lipschitz continuous, hence $U'(t)$ exists a.e.
In fact, assume $0<t_1\le t<t+\delta\le t_2<T$ and notice that
$|u_t(x,\tau)|\le C_1$ if $x\in\Omega$ and $\tau\in[t_1,t_2]$.
Choose $x_1,x_2$ such that $U(t+\delta)<u(x_1,t+\delta)+\delta$ and $U(t)<u(x_2,t)+\delta$.
Then there exist $\theta_1,\theta_2\in(0,1)$ such that
$$
U(t+\delta)-U(t) \begin{cases}
 \le u(x_1,t+\delta)-u(x_1,t)+\delta
 = \delta u_t(x_1,t+\delta\theta_1)+\delta 
 \le (C_1+1)\delta, \\
 \ge u(x_2,t+\delta)-u(x_2,t)-\delta
 = \delta u_t(x_2,t+\delta\theta_2)-\delta
 \ge -(C_1+1)\delta,
\end{cases}
$$
hence $U$ is locally Lipschitz.
We will show the inequality $U'(t)\ge(1-\eps)U^p(t)$ for all $t\in [t_0,T)$
such that $U'(t)$ exists.

Assume on the contrary that there exist
solutions $u_k$ with initial data $u_{0,k}$, $\|u_{0,k}\|_\infty\le A$,
and $\tau_{0,k}\le \tau_k<T(u_{0,k})$ such that 
$U_k'(\tau_k)$ exists,
\be{Ukdelta0}
U_k(\tau_{0,k})>k
\quad\hbox{ and }\quad
U_k'(\tau_k)<(1-\eps)U_k^p(\tau_k),
\ee
where $U_k(t)=\|u_k(\cdot,t)\|_\infty$.
Denoting $M_k(t):=\sup_{\tau\le t}U_k(\tau)$,
we claim that there exist $t_{0,k}\le t_k<T(u_{0,k})$
and $\delta_k\in(0,T(u_{0,k})-t_k)$ such that
\be{Ukdelta}
U_k(t_{0,k})>k,\ \
U_k(t_k)=M_k(t_k)
\,\hbox{ and }\,
 U_k(t_k+\delta)\le U_k(t_k)+\delta(1-\eps)U_k^p(t_k),\ \delta\in[0,\delta_k].
\ee 
Indeed, if $U_k(\tau_k)=M_k(\tau_k)$, then we may take $t_{0,k}=\tau_{0,k}$ and $t_k=\tau_k$ in view of \eqref{Ukdelta0}.
If $U_k(\tau_k)\ne M_k(\tau_k)$,
then there exists $t_k<\tau_k$ such that $U_k(t_k)=M_k(\tau_k)\ge U_k(\tau_{0,k})>k$,
hence $U_k(t)\le U_k(t_k)$ for all $t\in [t_k,\tau_k]$,
and \eqref{Ukdelta} is true with $t_{0,k}=\min(\tau_{0,k},t_k)$ and $\delta_k=\tau_k-t_k$.

In addition, if $k$ is large enough, then we also have $t_k\ge\xi$ for some $\xi>0$,
and, by parabolic regularity, we have
\be{uttCk}
C_k:=\sup_{\Omega\times[t_k,t_k+\delta_k]}|\partial^2_{tt}u_k|<\infty.
\ee
We claim that 
\be{ukxktk} 
(\exists x_k\in\Omega) \quad u_k(x_k,t_k)>U_k(t_k)-1 \ \hbox{ and }\ \partial_tu_k(x_k,t_k)\le(1-\eps/2)u_k^p(x_k,t_k).
\ee
Assume that \eqref{ukxktk} fails. 
Fix $k$ and choose 
$$\delta\in\Bigl(0,\min\Bigl\{\delta_k,\frac{\eps U_k^p(t_k)}{4C_k}\Bigr\}\Bigr)$$
and 
$$\eta\in\Bigl(0,\min\Bigl\{1,\frac18 \delta\eps U_k^p(t_k)\Bigr\}\Bigr)\ \hbox{ such that }\
\Bigl(1-\frac\eps2\Bigr)(U_k(t_k)-\eta)^p>\Bigl(1-\frac{3\eps}4\Bigr)U_k^p(t_k).$$
Then there exist $\tilde x_k\in\Omega$ such that
$$ u_k(\tilde x_k,t_k)>U_k(t_k)-\eta \ \hbox{ and }\ \partial_tu_k(\tilde x_k,t_k)>(1-\eps/2)u_k^p(\tilde x_k,t_k).$$
Consequently, the Mean Value Theorem, \eqref{uttCk}, the choice of $\delta,\eta$ and \eqref{Ukdelta} imply
$$ \begin{aligned}
 u_k(\tilde x_k,t_k+\delta)
&\ge u_k(\tilde x_k,t_k)+\delta(1-\eps/2)u_k^p(\tilde x_k,t_k)
  + \int_0^\delta (\delta-\theta)\partial^2_{tt}u_k(\tilde x_k,t_k+\theta)\,d\theta \\
&\ge U(t_k)-\eta +\delta(1-\eps/2)(U_k(t_k)-\eta)^p-C_k\delta^2/2 \\
&\ge U(t_k)+\delta(1-\eps)U_k^p(t_k)+\delta\eps U_k^p(t_k)/4-\eta-C_k\delta^2/2 \\
&> U(t_k)+\delta(1-\eps)U_k^p(t_k)  \ge U_k(t_k+\delta),
\end{aligned}$$
which is a contradiction.
Hence \eqref{ukxktk} is true. 

The proof is then concluded by the same argument as in the proof of Theorem~\ref{thm1a},
except that the equation for $v$ becomes 
$v_s-\Delta v=f(M_kv)/M_k^p$ instead of $v_s-\Delta v=v^p$
and we have $v_s(0,0)\le(1-\eps/2)v^p(0,0)$ instead of $v_s(0,0)\le 0$.
\end{proof}

\begin{remark} \rm
The proof of Theorem~\ref{thm1aN} shows that, given $c\in(0,1)$
there exists $K=K(A,f,\Omega,\eps,c)>0$ such that
$$\bigl(\|u_0\|_\infty\le A,\ U(t_0)>K,\ t\in [t_0,T),\ u(x,t)\ge cU(t)\bigr)
  \Longrightarrow u_t(x,t)\ge(1-\eps)f(u(x,t)).$$ 
In addition, we analogously 
 obtain the inequality $u_t(x,t)\le(1+\eps)f(u(x,t))$,
hence $|u_t(x,t)-f(u(x,t))|\le \eps f(u(x,t))$.
This property is similar to the ODE blow-up behavior
described in \cite{MZ,MZ00}.
\qed \end{remark}

\section{Proof of Propositions~\ref{thm2}-\ref{prop4}}
\label{SecProof3}

\begin{proof}[Proof of Proposition~\ref{thm2}]
Let $H=\frac12|\nabla u|^2+F(u)$ and set $Q:=\Omega\times (0,t_0]$.
Note that 
\be{regulH}
H\in BC(\overline Q)\cap C^{2,1}(Q)
\cap C^{1,0}(\overline\Omega\times (0,t_0])
\ee
by parabolic regularity
and that $K:=F(M(t_0))\ge\|H(0)\|_\infty$ by assumption~\eqref{hypFML}.
Assume for contradiction that the conclusion is not true.
Then there exist $\eta>0$ and $(x_0,t_0)\in Q$ such that 
$\tilde H:=H-K-\eta$ satisfies $\tilde H(x_0,t_0)>0$.
Set  $\tilde\Sigma:=Q\cap \{\tilde H>0\}$ and define the vector field
$$B:=|\nabla u|^{-2}(2f(u)\nabla u-\nabla H)\quad\hbox{on $\{|\nabla u|>0\}$.}$$
One can check (see \cite{Sp, FML} or the proof of \cite[Proposition 24.4a]{QS}) 
that 
\be{SestBUprofileLower4}
|\nabla u|^2\ge 2\eta\quad\hbox{and}\quad\tilde H_t-\Delta \tilde H\le B\cdot\nabla H 
\quad\hbox{on $\tilde\Sigma$}
\ee
and  (owing to the convexity of $\Omega$)
\be{SestBUprofileLower5}
\partial_\nu \tilde H\le 0\quad\hbox{on $\partial\Omega\times(0,t_0]$.}
\ee
It follows from \eqref{SestBUprofileLower4} that
$$\mathcal{P}\tilde H:=\tilde H_t-\Delta \tilde H -C|\nabla \tilde H| \le 0\quad\hbox{on $\tilde\Sigma$,}$$
where $C=2\eta^{-1/2}\sup_{Q}f(u)$.
To take care of the possible unboundedness of $\Omega$, we now set
$$\psi = \tilde H-\eps\phi,\qquad \phi=(n+1+C)t+(1+|x|^2)^{1/2}, $$
where $\eps>0$ is small enough so that $\psi(x_0,t_0)>0$.
By direct computation, we have
$\mathcal{P}\phi\ge 1$, hence
$$\mathcal{P}\psi\le \mathcal{P}\tilde H_t-\eps\mathcal{P}\phi<0
\quad\hbox{in $Q\cap \{\psi>0\}$}.$$
Since $\psi(\cdot,0)\le 0$ and $\psi\le 0$ in $\overline Q\cap\{|x|>R\}$ for $R=R(\eps)>0$ large, owing to \eqref{regulH},
$\psi$ attains its positive maximum over $\overline Q$ at some $(x_1,t_1)\in\overline\Omega\times(0,t_0]$.
If $x_1\in\Omega$, then at this point we have $\psi_t\ge 0$, $\Delta \psi\le 0$, $\nabla \psi=0$, hence
$0\le \mathcal{P}\psi<0$, which is impossible.
Therefore, $x_1\in\partial\Omega$. But at this point, recalling  \eqref{regulH}, we have $\partial_\nu\psi>0$ by the Hopf Lemma.
However, $\partial_\nu\psi=\partial_\nu \tilde H-\eps (x_1\cdot\nu)(1+|x_1|^2)^{-1/2}\le 0$
by \eqref{SestBUprofileLower5} and the convexity of $\Omega$ (assuming $0\in\Omega$ without loss of generality).
This is a contradiction, which concludes the proof.
\end{proof}

\begin{proof}[Proof of Proposition~\ref{prop3}.]
(i) Taking supremum in \eqref{hypFML3}, we get $F(U(t_1))\le F(U(t_0))$.
Since $F$ is an increasing bijection from $[0,\infty)$ onto itself in view of \eqref{hypf},
we deduce that $U(t_1)\le U(t_0)$. 

\smallskip

(ii) The sufficiency is given by Proposition~\ref{thm2}, 
whereas the necessity follows from assertion (i).
\end{proof}

\begin{remark} \label{gapFML} \rm
At~p.433,~l.13 of \cite{FML}, the proof requires that $J:=\frac12|\nabla u|^2+F(u)-F(U(t_0))\le 0$ at any point of $\Omega\times(0,t_0)$ where $\nabla u=0$. 
But, since $F$ is an increasing function, this is precisely equivalent to $U(t_0)=M(t_0)$.
In fact, earlier in the paper, at~p.425, it is claimed that ``by the maximum principle, [\dots] $U(t)$ is monotone increasing in $t$'',
which would indeed imply $U(t_0)=M(t_0)$, but this statement is not correct.
\qed \end{remark}

\smallskip

\begin{proof}[Proof of Proposition~\ref{prop4}.]
(i) We denote $J=\Omega$, $r=x\in J$ if $n=1$, and
$J=[0,R)$,  $r=|x|\in J$ if $n\ge 2$ and $u_0$ is radial.
We can assume $u_0\not\equiv0$.
We know that, for all $t\in(0,T)$, 
\be{Ut0}
\hbox{$\{r\in J;\ u_r(r,t)=0\}$ is finite}
\ee
and that its cardinal $N(t)$ is
a nonincreasing function of $t$ and drops if $u_r$ has a multiple zero  (see, e.g.,~\cite[Section~52.8]{QS}).
Consequently, given $t_0\in(0,T)$, there exists $\eta>0$ such that $u_r(\cdot,t)$ has only simple zeros for all $t\in I^-\cup I^+$,
where $I^-=(t_0-\eta,t_0)$ and $I^+=(t_0,t_0+\eta)$.
As a consequence of the implicit function theorem, it follows in particular that the set of zeros of 
$u_r(\cdot,t)$ for $t\in I^-\cup I^+$ consists of finite numbers $k^\pm\ge 1$ of curves $r^\pm_i\in C^1(I^\pm)$ 
with $r^\pm_1(t)<\dots<r^\pm_{k^\pm}(t)$. 
Moreover, the limits $r^\pm_i(t_0)=\lim_{t\to t_0^\mp}r^\pm_i(t)$ exist
(otherwise, one of these curves would satisfy $\alpha_1:=\liminf \gamma(t)<\limsup\gamma(t)=:\alpha_2$ as $t\to t_0^+$ or $t_0^-$, which
would lead to $u_r(\cdot,t_0)\equiv 0$ on $[\alpha_1,\alpha_2]$: a contradiction to \eqref{Ut0}).

We have 
\be{Ut1}
U(t)=\max_{1\le i\le k^\pm} u(r^\pm_i(t),t),\quad t\in I^\pm,
\ee
and this remains true for $t=t_0$ by the continuity of the functions 
$U(t)$ and $u(r^\pm_i(t),t)$ as $t\to t_0^\pm$.
Denote
$$\Sigma^\pm=\bigl\{i\in \{1,\dots,k^\pm\};\ u(r^\pm_i(t_0),t_0)=U(t_0)\bigr\}\ne\emptyset.$$
By continuity, there exists $\eta_1\in(0,\eta)$ such that
$u(r^\pm_i(t),t)<U(t)$ 
if $i\not\in \Sigma^\pm$ and $t\in I^\pm\cap (t_0-\eta_1,t_0+\eta_1)$.
By \eqref{Ut1}, we deduce that
\be{Ut4}
U(t)=\max_{i\in\Sigma^\pm} u(r^\pm_i(t),t),\quad t\in I^\pm\cap (t_0-\eta_1,t_0+\eta_1).
\ee

We claim that $U'_\pm(t_0)$ exist and are given by
\be{diffpm}
U'_+(t_0)=\sigma^+:=\max_{i\in \Sigma^+} u_t(r^+_i(t_0),t_0),\quad U'_-(t_0)=\sigma^-:=\min_{i\in \Sigma^-} u_t(r^-_i(t_0),t_0)\le \sigma^+.
\ee
Let $i\in \Sigma^\pm$. 
We have
$$u(r^\pm_i(t),t)=u(r^\pm_i(t),t_0)+\int_{t_0}^t u_t(r^\pm_i(t),s)ds
\le U(t_0)+\int_{t_0}^t u_t(r^\pm_i(t),s)ds,\quad t\in I^\pm$$
hence, by the continuity of $u_t$,
$$
u(r^\pm_i(t),t)\le U(t_0)+(1+o(1))u_t(r^\pm_i(t_0),t_0)(t-t_0),\quad t\to t_0^\pm.
$$
By \eqref{Ut4} we deduce that
\be{mi2}
U(t)\le U(t_0)+(1+o(1))\sigma^\pm(t-t_0),\quad t\to t_0^\pm.
\ee
On the other hand, we have 
$$
U(t)\ge u(r^\pm_i(t_0),t)=u(r^\pm_i(t_0),t_0)+(1+o(1))u_t(r^\pm_i(t_0),t_0)(t-t_0),\quad t\to t_0^\pm,\quad i\in \Sigma^\pm,
$$
hence
$$U(t)\ge U(t_0)+(1+o(1))\sigma^\pm(t-t_0),\quad t\to t_0^\pm.$$
This along with \eqref{mi2} proves the claim, hence the first part of assertion (i).

\smallskip

Let us show the second part of assertion (i).
Set $E:=\{t\in(0,T):U'_+(t)>U'_-(t)\}$ and fix $\eps>0$.
The set $D_\eps:=\{t\in(\eps,T):\hbox{$u_r(\cdot,t)$ has a multiple zero}\}$
is finite. Fix an open interval $I\subset(\eps,T)\setminus D_\eps$. As above,
the set of zeroes of $u_r$ in $J\times I$ consists of finite number $k$ of curves
$r_i\in C^1(I)$ with $r_1(t)<r_2(t)<\dots<r_k(t)$. Set $\varphi_i(t):=u(r_i(t),t)$
and notice that $\varphi'_i(t)=u_t(r_i(t),t)$ due to $u_r(r_i(t),t)=0$.
Now \eqref{diffpm} implies $E\cap I\subset\bigcup_{i\ne j}A_{ij}$,
where 
$$ A_{ij}:=\{t\in I: \varphi_i(t)=\varphi_j(t)\hbox{ and }\varphi_i'(t)>\varphi_j'(t)\} $$
is obviously isolated in $I$. Consequently, $E$ is at most countable.

\smallskip

(ii) Assume that $U'_-(t_0)<0$.
By \eqref{diffpm}, there exists $r_0\in J$ 
such that $u(r_0,t_0)=U(t_0)$ and $u_t(r_0,t_0)<0$.
Moreover $r_0>0$ if $n\ge 2$.
By \eqref{Ut0} there exists $\eps>0$ such that $u_r(r,t_0)<0=u_r(r_0,t_0)$ for $r_0<r<r_0+\eps$.
Set 
$$\hat J(r)=\frac12 |u_r(r,t_0)|^2+F(u(r,t_0))-F(U(t_0)).$$
Then,  by continuity,  
we have
$$\hat J_r=u_r(u_{rr}+f(u))=u_r \bigl(u_t-(n-1)r^{-1}u_r\bigr)>0,\quad\hbox{for $r-r_0>0$ small}.$$
Consequently $\hat J(r)>\hat J(r_0)=0$ for 
$r-r_0>0$ small, hence \eqref{conclFML4} fails.
\end{proof}


\noindent{\bf Declarations.} The authors report that there are no competing interests to declare. This manuscript has no associated data.

\smallskip

\noindent{\bf Acknowledgement.}
The first author was supported in part by the Slovak Research and Development Agency
under the contract No.~APVV-23-0039 and by VEGA grant 1/0245/24.

\end{document}